\documentclass[12pt]{amsart}
\usepackage[margin=1.25in]{geometry}

\usepackage{amsmath, amscd}
\usepackage{mathrsfs}
\usepackage{amsfonts}
\usepackage{latexsym,epsfig}
\usepackage{mathabx}
\usepackage{amsmath, amscd, amsthm}
\usepackage[pdftex]{color}
\usepackage{comment}
\usepackage{marginnote}
\usepackage[colorlinks,citecolor=blue,linkcolor=red]{hyperref}

\newcommand{\Out}[0]{\mathrm{Out}}

\newcommand{\PP}[0]{\mathcal{P}}

\newcommand{\cone}{\mathrm{cone}}

\usepackage{hyperref}

\usepackage{amssymb,amsfonts}
\usepackage[all,arc]{xy}
\usepackage{enumerate}
\usepackage{mathrsfs}
\usepackage{graphicx}

\theoremstyle{definition}

\theoremstyle{remark}

\theoremstyle{theorem} 
\newtheorem{theorem}{Theorem}[section]
\newtheorem{lemma}[theorem]{Lemma}
\newtheorem{proposition}[theorem]{Proposition}
\newtheorem{corollary}[theorem]{Corollary}
\newtheorem{conjecture}[theorem]{Conjecture}

\theoremstyle{definition}
\newtheorem{definition}[theorem]{Definition}

\newtheorem{remark}[theorem]{Remark}

\newcommand{\C}{\mathcal{C}}

\newcommand{\Mod}{\mathrm{Mod}}

\renewcommand{\P}{\mathcal{P}}
\renewcommand{\PP}{\mathbb{P}}
\newcommand{\ev}{\mathrm{ev}}
\newcommand{\GammaT}{\widetilde{\Gamma}}
\newcommand{\cay}{\mathrm{C}}

\begin{document}

\title{Counting loxodromics for hyperbolic actions}

\author[I. Gekhtman]{Ilya Gekhtman}
\address{Department of Mathematics\\ 
Yale University\\ 
10 Hillhouse Ave\\ 
New Haven, CT 06520, U.S.A\\}
\email{\href{mailto:ilya.gekhtman@yale.edu}{ilya.gekhtman@yale.edu}}

\author[S.J. Taylor]{Samuel J. Taylor}
\address{Department of Mathematics\\ 
Yale University\\ 
10 Hillhouse Ave\\ 
New Haven, CT 06520, U.S.A\\}
\email{\href{mailto:s.taylor@yale.edu}{s.taylor@yale.edu}}

\author[G. Tiozzo]{Giulio Tiozzo}
\address{Department of Mathematics\\ 
Yale University\\ 
10 Hillhouse Ave\\ 
New Haven, CT 06520, U.S.A\\}
\email{\href{mailto:giulio.tiozzo@yale.edu}{giulio.tiozzo@yale.edu}}

\date{\today}

\begin{abstract}
Let $G \curvearrowright X$ be a nonelementary action by isometries of a hyperbolic group $G$ on a hyperbolic metric space $X$. We show that the set of elements of $G$ which act as loxodromic isometries of $X$ is generic. That is, for any finite generating set of $G$, the proportion of $X$--loxodromics in the ball of radius $n$ about the identity in $G$ approaches $1$ as $n \to \infty$. We also establish several results about the behavior in $X$ of the images of typical geodesic rays in $G$; for example, we prove that they make linear progress in $X$ and converge to the Gromov boundary $\partial X$. Our techniques make use of the automatic structure of $G$, Patterson--Sullivan measure on $\partial G$, and the ergodic theory of random walks for groups acting on hyperbolic spaces. 
We discuss various applications, in particular to $\Mod(S)$, $\Out(F_N)$, and right--angled Artin groups.
\end{abstract}

\maketitle

\section{Introduction}

Let $G$ be a hyperbolic group with a fixed finite generating set $S$. Then $G$ acts by isometries on its associated Cayley graph $\cay_S(G)$, which itself is a geodesic hyperbolic metric space. This cocompact, proper action has the property that each infinite order $g\in G$ acts as a \emph{loxodromic isometry}, i.e. with sink--source dynamics on the Gromov boundary $\partial G$. Much geometric and algebraic information about the group $G$ has been learned by studying the dynamics 
of the action $G \curvearrowright \cay_S(G)$, beginning with the seminal work of Gromov \cite{Gromov}. However, deeper facts about the group $G$ can often be detected by investigating actions $G \curvearrowright X$ which are specifically constructed to extract particular information about $G$. Important examples include the theory of JSJ decompositions of $G$ \cite{rips_sela,bowditch1998cut}, or more generally actions on trees arising from splittings of $G$ \cite{Serre}, $G$ acting on the quasi--trees of Bestvina--Bromberg--Fujiwara \cite{BBF}, or $G$ acting on coned--off versions of $\cay_S(G)$ as appearing in the theory of relatively hyperbolic groups \cite{farb1998relatively, osin2006relatively}, hyperbolically embedded subgroups \cite{DGO}, and hyperbolic Dehn surgery \cite{osin2007peripheral, groves2008dehn}. Several other examples appear in Section \ref{sec:apps}. 

In this paper, we are interested in the typical behavior of elements of $G$ with respect to an arbitrary action $G \curvearrowright X$. Working in this level of generality has applications to the natural actions of $G$ that one regularly encounters, regardless of whether the action is nonproper or distorted.

Let $G$ be a hyperbolic group which acts by isometries on a hyperbolic metric space $X$. A choice of a finite generating set $S$
determines a word metric on $G$, and we denote as $B_n$ the ball of radius $n$ about $1 \in G$.

We say that a subset $A \subset G$ is \emph{generic} if the proportion of elements of word length at most $n$ which belong to $A$ 
tends to $1$ as $n \to \infty$, i.e.
\[
\lim_{n\to \infty} \frac{\#(A \cap B_n)}{\#B_n} = 1.
\]

An element $g \in G$ is said to be \emph{loxodromic} if it has exactly two fixed points on the boundary of $X$, one attracting and the other repelling. This is equivalent to the condition that there is a positive constant $\epsilon = \epsilon(g)$ such that $ d_X(x,g^nx) \ge \epsilon \cdot |n|$ for all $n \in \mathbb{Z}$, where $|g|$ denotes the word length of $G$ with respect to $S$. 


Our first main theorem is the following:

\begin{theorem}[Genericity of loxodromics] \label{cor:intro_gen_lox}
Let $G$ be a hyperbolic group with a nonelementary action by isometries on a separable, hyperbolic metric space $X$. Then $X$--loxodromics are generic on $G$, i.e.
\[
\frac{\#\{g \in B_n : g \; \mathrm{is} \; X - \mathrm{loxodromic} \}}{\#B_n} \to 1,
\]
as $n\to \infty$.
\end{theorem}




Recall that two loxodromic isometries $f$ and $g$ of $X$ are \emph{independent} 
if their fixed point sets in $\partial X$ are disjoint, and
the action $G\curvearrowright X$ is \emph{nonelementary} if $G$ contains two independent loxodromic isometries of $X$\footnote{Actions satisfying this condition are sometimes called \emph{of general type} \cite{caprace2015amenable}.}.

We actually prove the stronger result that $X$--loxodromics are generic with respect to counting in spheres $S_n$ in the Cayley graph of $G$. 



Our results demonstrate that much of the typical geometry of $G$ is preserved under the action $G \curvearrowright X$. 
Let us recall that the choice of a generating set $S$ determines a natural boundary measure $\nu$
on $\partial G$ which is called the  \emph{Patterson--Sullivan measure}  \cite{coornaert1993mesures}.
Intuitively, this measure is defined by taking weak limits of the uniform distribution on balls $B_n$ as $n \to \infty$ (with the previous definition, it is only defined up to a multiplicative constant, but we actually fix a normalization: see Section \ref{sec:hyperbolic}).

\subsection{Typical geodesic rays in $G$}

Our next collection of results concerns the behavior of $\nu$--typical geodesic rays of $G$ and the paths they determine in the space $X$ under the orbit map $G \to X$. That the images of geodesics in $G$ have any controlled behavior when projected to $X$ is particularly surprising at this level of generality. A few examples to bear in mind are $G$ acting on a locally infinite hyperbolic graph, $G$ acting on one of its nonelementary hyperbolic quotient groups, or $G$ acting on $\mathbb{H}^n$ with dense orbits.

We first show that typical geodesic rays of $G$ converge to the boundary in $X \cup \partial X$:


\begin{theorem} [Convergence to the boundary of $X$] \label{th:intro_1}
Let $G$ be a hyperbolic group with a nonelementary action by isometries on a separable, hyperbolic metric space $X$.
For every $x \in X$ and $\nu$--almost every $\eta \in \partial G$, if $(g_n)_{n\ge0}$ is a geodesic ray in $G$ converging to $\eta$, then the sequence $g_nx$ in $X$ converges to a point in the boundary $\partial X$.
\end{theorem}


For $x\in X$, let $\Phi = \Phi_x \colon G \to X$ given by $\Phi(g) = gx$ denote the orbit map of the action $G \curvearrowright X$.
Theorem \ref{th:intro_1} implies the existence of a $\nu$-measurable, $G$-equivariant map 
 $\partial \Phi \colon \partial G \to \partial X$ from the Gromov boundary of $G$ to the Gromov boundary of $X$.

In fact, if we define
\[
\partial^X G = \{\eta \in \partial G : \Phi([1,\eta)) \subset X \text{ converges to a point in } \partial X\},
\]
then $\Phi$ extends to a map $\partial \Phi \colon \partial ^X G \to \partial X$ and the set $\partial^X G$ has full $\nu$-measure by Theorem \ref{th:intro_1}.

In addition, we show that for almost every $\eta \in \partial G$, the path $\Phi([1,\eta))$ makes linear progress in the space $X$:

\begin{theorem}[Positive drift] \label{th:intro_2}
Let $G$ be a hyperbolic group with a nonelementary action by isometries on a separable, hyperbolic metric space $X$.
Then there exists $L > 0$ such that for every $x \in X$ and $\nu$--almost every $\eta \in \partial G$, if $(g_n)_{n\ge0}$ is a geodesic in $G$ converging to $\eta$, then 
\[
\lim_{n \to \infty} \frac{d_X(x, g_n x)}{n} = L >0.
\]
\end{theorem}


Theorem \ref{th:intro_1} and Theorem \ref{th:intro_2} are proven using a Markov chain (introduced by Calegari--Fujiwara \cite{calegari2010combable}) on the directed graph $\Gamma$ which parameterizes a geodesic combing of $G$. That is, directed paths in $\Gamma$ evaluate to geodesics in the group $G$. 
Let $\mathbb{P}$ be the corresponding Markov measure on the space of sample paths $(w_n)$ which begin at the ``identity'' vertex of $\Gamma$. (See Section \ref{sec:Markov_action} for details.) By combining Theorem \ref{th:intro_1} and Theorem \ref{th:intro_2} we show that along a $\mathbb{P}$--typical path in $\Gamma$ one encounters elements of $G$ which are loxodromic for the action on $X$ and whose translation lengths grow linearly. 
Recall that the \emph{(stable) translation length of $g$} with respect to its action on $X$ is
\[
\tau_X(g) = \liminf_{n \to \infty} \frac{d_X(x,g^nx)}{n},
\]
which is well-defined, independent of $x \in X$, and an element $g \in G$ is $X$--loxodromic if and only if $\tau_X(g) >0$.

\begin{theorem}[Linear growth of $X$--translation lengths] \label{th:intro_Mark}
Let $G$ be a hyperbolic group with a nonelementary action by isometries on a separable hyperbolic metric space $X$.
Then there is an $L_1 >0$ such that 
\[
\PP \big(\tau_X(w_n) \ge L_1 n \big) \to 1,
\]
as $n \to \infty$.
\end{theorem}
\noindent In fact, we show that the constant $L_1$ appearing in Theorem \ref{th:intro_Mark} can be taken to be $L - \epsilon$ for any $\epsilon > 0$, 
where $L$ is the drift constant of Theorem \ref{th:intro_2}.


\subsection{Genericity in $G$ and the action $G \curvearrowright X$}
Our third collection of results concerns the typical behavior of elements $g\in G$ with respect to counting in the balls $B_n \subset G$. We emphasize that our genericity results also hold for counting in spheres $S_n$ (i.e. elements of word length exactly $n$) which gives a priori stronger information. 

First, we show that for a fixed $x \in X$, the ratio of the displacement $d_X(x,gx)$ to the word length of $g$ is uniformly bounded from below for a generic subset of $G$.

\begin{theorem}[Genericity of positive drift] \label{th: intro_gen_drift}
Let $G$ be a hyperbolic group with a nonelementary action by isometries on a separable hyperbolic metric space $X$.
Then there exists $L_1 > 0$ such that 
$$\frac{\#\{ g \in B_n \ : \ d_X(gx,x) \geq L_1|g| \}}{\#B_n} \to 1 \qquad \textup{as }n \to \infty.$$
\end{theorem} 


By combining Theorem \ref{th: intro_gen_drift} with estimates for the Gromov product between generic elements we show that, generically, the translation length of $g$ with respect to the action $G \curvearrowright X$ grows linearly with $|g|$.

\begin{theorem}[Genericity of linear growth] \label{th:intro_gen_trans}
Let $G$ be a hyperbolic group with a nonelementary action by isometries on a separable hyperbolic metric space $X$.
Then there is an $L_1 >0$ so that 
\[
\frac{\#\{g \in B_n : \tau_X(g) \ge L_1|g| \}}{\#B_n} \to 1,
\]
as $n\to \infty$.
\end{theorem}

Just as above, the constant $L_1$ appearing in Theorem \ref{th: intro_gen_drift} and Theorem \ref{th:intro_gen_trans} can be taken to be $L - \epsilon$
for any $\epsilon > 0$, where $L$ is the drift constant of Theorem \ref{th:intro_2}. \\

Since $\tau_X(g) >0$ if and only if $g$ is loxodromic for the action $G \curvearrowright X$, Theorem \ref{th:intro_gen_trans} immediately implies that $X$--loxodromics are generic in $G$, proving Theorem \ref{cor:intro_gen_lox}.

\subsection{Methods and connections to previous results}

Counting problems for discrete groups have a long history; in particular, starting with Margulis' thesis \cite{margulis}, 
much attention has been devoted to counting orbit points of a lattice in a Lie group, 
with respect to the Riemannian metric in the Lie group. This is closely related to counting geodesics of a certain length in the quotient space. 
 

However, few works have addressed the counting with respect to the \emph{word metric} on $G$. 
For example, Pollicott and Sharp \cite{pollicott1998comparison} compare word length in a cocompact lattice of $\mathrm{Isom}(\mathbb{H}^n)$  to distance between orbit points in $\mathbb{H}^n$.
In a different vein, Calegari and Fujiwara \cite{calegari2010combable} study the generic behavior of a bicombable function on a hyperbolic group and, in particular, establish a Central Limit Theorem for such functions. Finally, Wiest \cite{wiest2014genericity} recently showed that if a group $G$ satisfies a weak automaticity condition and the action $G \curvearrowright X$ on a hyperbolic space $X$ satisfies a strong ``geodesic word hypothesis,'' then the loxodromics make up a \emph{definite proportion} of elements of the $n$ ball (when counting with respect to certain normal forms). This geodesic word hypothesis essentially requires geodesics in the group $G$, given by the normal forms, to project to unparameterized quasigeodesics in the space $X$ under the orbit map. 

Another way of counting is to run a random walk on the Cayley graph of $G$ and count with respect to the $n$-step distribution of such random walk.
Since sample paths are generally not geodesics, this counting is also different from counting in balls. However, 
in this case many more results are known: for instance, Rivin \cite{rivin2008walks} and Maher \cite{Maher} proved that pseudo-Anosov mapping classes 
are generic with respect to random walks in the mapping class group.
Moreover, if a random walk converges almost surely
to $\partial X$, then it defines a \emph{harmonic measure} on $\partial X$ which is the hitting measure at infinity of the walk. 
See, among others, \cite{kaimanovich1994poisson, calegari2015statistics, MaherTiozzo}. 

In general, it is hard to compare the hitting measure for a random walk with the Patterson--Sullivan measure, and that makes the 
two types of counting different. 
In our case though, the theory of geodesic combings of hyperbolic groups \cite{cannon1984combinatorial} tells us that counting with respect 
to balls is equivalent to counting paths in a certain finite graph. Further, there exists a Markov chain on the graph whose $n$-step distribution produces essentially 
the uniform distribution on balls \cite{calegari2010combable}.  \\




One consequence of our study is that (a particular normalization of) the Patterson--Sullivan measure $\nu$ on $\partial G$
decomposes as a countable sum of harmonic measures associated to random walks on $G$, in the following sense. 



\begin{proposition} \label{prop:intro_sum}
There exists a finite collection of measures $\nu_1, \dots, \nu_r$ on $\partial G$, which are harmonic measures for random walks on $G$ with finite 
exponential moment, and such that the Patterson--Sullivan measure $\nu$ can be written as 
\[
\nu = \sum_{g \in G} a_g \ g_* \nu_{i(g)},
\]
where each $a_g$ is a non-negative, real coefficient, and $i(g) \in \{1, \dots, r \}$.
\end{proposition}

One should compare Proposition \ref{prop:intro_sum} with the main theorem of Connell--Muchnik \cite{connell2007harmonicity},
who show that for certain actions of hyperbolic groups, the Patterson--Sullivan measure is actually the harmonic measure for some random walk on $G$. We note that their result does not apply in our setting (the Gromov boundary of a hyperbolic group) since the Patterson--Sullivan measure $\nu$ 
does not necessarily have the property that the Radon--Nikodym derivative of $g_*\nu$ with respect to $\nu$ is continuous.
See also Remark 2.15 of \cite{gouezel2015entropy}. 

\subsection{Applications} \label{sec:apps}
We now collect some immediate applications of our main results. 

\subsubsection{Splittings and quotients of hyperbolic groups}
Applying Theorem \ref{th:intro_gen_trans} directly to the action $G \curvearrowright C_S(G)$ of a hyperbolic group on its Cayley graph, we see that generic elements of $G$ are infinite order and their translation lengths grow linearly in word length. 
The main point of our results is that they apply to a much more general setting.

For example, recall that a splitting of a group $G$ is an action of $G$ on a simplicial tree $T$ -- this is equivalent to realizing $G$ as the fundamental group of a graph of groups \cite{Serre}. The splitting is 
minimal if there is no invariant subtree. For $g \in G$, we denote by $g^*$ 
the shortest representative of the conjugacy class of $g$ in $T/G$.

\begin{proposition}
Suppose that $G \curvearrowright T$ is a minimal splitting of a hyperbolic group $G$ such that $T$ has at least $3$ ends. Then the set of elements which are not conjugate into a vertex stabilizer is generic.

Moreover,  there is an $L >0$ such that the set of $g \in G$ having the property that 
\[
\#(\mathrm{edges} \ \mathrm{crossed} \ \mathrm{by} \ g^* ) \ge L |g|
\]
is generic in $G$.
\end{proposition}

Similar statements can be made about other useful actions of $G$ on hyperbolic spaces. These include the quasi--trees of Bestvina--Bromberg--Fujiwara \cite{BBF} or hyperbolic graphs obtained by coning--off uniformly quasiconvex subsets of the hyperbolic group $G$ \cite{KapRaf}.
We state one further general application, which follows directly from Theorem \ref{th: intro_gen_drift} and Theorem \ref{th:intro_1}.

\begin{theorem}[Epimorphisms are generically bilipschitz]
Let $\phi \colon G \to H$ be a surjective homomorphism between nonelementary hyperbolic groups. Then there is an $L >0$ such that the set of $g \in G$ for which 
\[
|\phi(g)| \ge L |g|
\]
is generic in $G$.

Moreover, there is a subset $\partial^H G$ of $\partial G$ with $\nu(\partial^H G) = 1$ and a boundary map $\partial \phi \colon \partial^H G \to \partial H$ extending the homomorphism $\phi \colon G \to H$.
\end{theorem}

\subsubsection{Mapping class groups and $\mathrm{Out}(F_N)$}
Our main result on the genericity of loxodromics is in part motivated by a long-standing conjecture about the mapping class group of an orientable surface $S$ with $\chi(S) <-1$. For background material on mapping class groups see, for example, \cite{FM}.

\begin{conjecture}[{\cite[Conjecture 3.15]{Farbproblems}}] \label{conj_farb}
Let $\Mod(S)$ be the mapping class group of an orientable surface $S$ with $\chi(S)<-1$. Then pseudo--Anosov mapping classes are  generic in $\Mod(S)$.
\end{conjecture}


While Conjecture \ref{conj_farb} seems at the moment out of reach, our main result does imply the corresponding statement for hyperbolic subgroups of $\Mod(S)$. Recall that $\Mod(S)$ is \emph{not} itself hyperbolic and so the techniques of this paper do not directly apply. We say that a subgroup $G$ of $\Mod(S)$ is irreducible if no finite index subgroup of $G$ fixes a multicurve on $S$.

\begin{theorem}[Genericity in $\Mod(S)$] \label{th:pA_gen}
Let $G$ be a nonelementary hyperbolic group which is an irreducible subgroup of $\Mod(S)$. Then pseudo--Anosov mapping classes are  generic in $G$ with respect to any generating set of $G$. 
\end{theorem}

\begin{proof}
The mapping class group $\Mod(S)$ acts by isometries on the curve complex $\C(S)$ of the surface $S$, which is hyperbolic by \cite{MM1}. Since $G$ is an irreducible subgroup of $\Mod(S)$ which is not virtually cyclic, $G$ has a nonelementary action on $\C(S)$; this follows from the subgroup structure theorems of \cite{BLM, Iv2}. Since the loxodromics of the action $\Mod(S) \curvearrowright \C(S)$ are exactly the pseudo-Anosov mapping classes, again by \cite{MM1}, the result follows from Theorem \ref{cor:intro_gen_lox}.
\end{proof}

We remark that hyperbolic, irreducible subgroups of $\Mod(S)$, i.e. those subgroups to which Theorem \ref{th:pA_gen} applies, are abundant. For example, there are several constructions of right--angled Artin subgroups of $\Mod(S)$ \cite{CLM, Kobraag} and such subgroups are well-known to contain a variety of hyperbolic subgroups (see Section \ref{sec:RAAGs}). 

For a second source of examples, let $S$ be a closed surface of genus at least $2$ and set $\mathring{S} = S \setminus p$ for some $p\in S$. Recall that there is a natural map $\Mod(\mathring{S}) \to \Mod(S)$ and for $h \in \Mod(S)$, the preimage of $\langle h \rangle$ under this map is exactly the subgroup $\pi_1(M_h) \le \Mod(\mathring{S}) $, the fundamental group of the mapping torus of the homeomorphism $h \colon S \to S$ \cite{birman1969mapping}. If $h$ is pseudo-Anosov, then Thurston's hyperbolization theorem for $3$--manifolds fibering over the circle implies that $M_h$ is hyperbolic \cite{thurston1998hyperbolic3}. In particular, $\pi_1(M_h)$ is hyperbolic. Theorem \ref{th:pA_gen} implies that for any pseudo-Anosov $h \in \Mod(S)$ and any generating set of $\pi_1(M_h)$, pseudo-Anosov mapping classes are generic in the subgroup $\pi_1(M_h) \le \Mod(\mathring{S})$. More interestingly, this is true for \emph{any} embedding of $\pi_1(M_h)$ into a mapping class group, so long as the image does not virtually fix a multicurve on the surface.

There is also a direct analogue for hyperbolic subgroups of $\Out(F_N)$. Similar to the discussion for mapping class groups, there are various techniques for constructing hyperbolic subgroups of $\Out(F_N)$. See for example \cite{Tayl1}. We say that a subgroup $G$ of $\Out(F_N)$ is irreducible if no finite index subgroup of $G$ fixes a 
free factor of $F_N$.

\begin{theorem}[Genericity in $\Out(F_N)$]
Let $G$ be a nonelementary hyperbolic group which is an irreducible subgroup of $\Out(F_N)$. Then fully irreducible automorphisms are generic in $G$ with respect to any generating set of $G$.

Moreover, if $G$ is also not contained in a mapping class subgroup of $\Out(F_N)$, then atoroidal, fully irreducible automorphisms are generic in $G$. 
\end{theorem}

\begin{proof}
We mimic the proof of Theorem \ref{th:pA_gen} using the action of $\Out(F_N)$ on two free group analogues of the curve complex. First, since $G \le \Out(F_N)$ is irreducible and not virtually cyclic, the main result of \cite{handel2013subgroup} implies that $G$ contains ``independent'' fully irreducible automorphisms. This implies that $G$ has a nonelementary action on $\mathcal{FF}_N$, the free factor complex of $F_N$, which is hyperbolic by \cite{BF14}. We then apply Theorem \ref{cor:intro_gen_lox} to obtain the first part of the theorem since the loxodromic isometries of $\mathcal{FF}_N$ are exactly the fully irreducible automorphisms \cite{BF14}.

To get the moreover statement, we use the action of $\Out(F_N)$ on $\mathcal{CS}_N$, the co-surface graph of $F_N$. By \cite{DT3}, $\mathcal{CS}_N$ is a hyperbolic graph and the loxodromic elements of the action $\Out(F_N) \curvearrowright \mathcal{CS}_N$ are exactly the atoroidal fully irreducible automorphisms. By \cite{uyanik2015generalized}, $G \le \Out(F_N)$ must contain an atoroidal element of $\Out(F_n)$ for otherwise $G$ is contained in a mapping class subgroup of $\Out(F_N)$, contrary to our hypothesis. From this, it follows easily that the action $G \curvearrowright \mathcal{CS}_N$ is nonelementary and so another application of  Theorem \ref{cor:intro_gen_lox} completes the proof.
\end{proof}

\subsubsection{Special hyperbolic groups and loxodromics in RAAGs} \label{sec:RAAGs}
As a final application, we recall the celebrated result of Agol, who building of work of Wise and his collaborators showed that every hyperbolic cubulated group $G'$ has a finite index subgroup $G$ which embeds into a right-angled Artin group $A(\Gamma)$ \cite{AgolVHC}. Following work of Haglund and Wise, a finitely generated group which embeds into a right--angled Artin group $A(\Gamma)$ is called \emph{special} \cite{HW1}. For additional information of RAAGs and special groups see, for example, \cite{wise2012riches}.

As $A(\Gamma)$ is itself a CAT$(0)$ group, an important role is played by its elements which act by rank--$1$ isometries on its associated CAT$(0)$ cube complex, its so-called Salvetti complex. Following \cite{KK3}, we call such elements of $A(\Gamma)$ \emph{loxodromic}.  Equivalently, loxodromics elements of $A(\Gamma)$ are characterized algebraically by having cyclic centralizer in $A(\Gamma)$ \cite{Servatius, BestvinaFujiwara} and geometrically as being the Morse elements of $A(\Gamma)$ \cite{BC}. 
Being of geometric interest, one might ask how frequently elements of a hyperbolic cubulated group $G$ get mapped to loxodromics of $A(\Gamma)$ via the embeddings provided by Agol's theorem, and our next application shows that this is indeed generically the case. Note this is in contrast to the well--known result that the only subgroups of $A(\Gamma)$ for which \emph{each} nontrivial element is loxodromic are free subgroups \cite{CLM, KK3,KMT}.

\begin{theorem}[Genericity in hyperbolic, special groups]
Let $G$ be a nonelementary 
hyperbolic group and $\phi \colon G \to A(\Gamma)$ be an injective homomorphism into a right--angled Artin group $A(\Gamma)$. Then either $\phi$ is conjugate to a homomorphism into $A(\Lambda)$ for a subgraph $\Lambda \le \Gamma$, or 
\[
\frac{\#\{g \in B_n : \phi(g) \; \mathrm{is} \; \mathrm{loxodromic} \:  \mathrm{in} \: A(\Gamma) \}}{\#B_n} \to 1.
\]
\end{theorem}

\begin{proof}
By Theorem 5.2 of \cite{BC}, if $\phi(G)$ is not conjugate into $A(\Lambda)$ for some proper subgraph $\Lambda \le \Gamma$, then $\phi(G)$ contains at least $2$ independent loxodromic elements. This implies that the induced action of $G$ on the extension graph $\Gamma^e$ is nonelementary. Since $\Gamma^e$ is hyperbolic and the loxodromic isometries of $A(\Gamma) \curvearrowright \Gamma^e$ are exactly the loxodromic elements of $A(\Gamma)$ (\cite{KK3}), applying Theorem \ref{cor:intro_gen_lox} to the action $G \curvearrowright \Gamma^e$ completes the proof.
\end{proof}

\subsection{Summary of paper}
In Section \ref{sec:counting_graphs} and Section \ref{sec:Mchains} we present the theory of Markov chains on directed graphs that will be needed for our counting arguments.
In Section \ref{sec:markov_to_walk} we describe a process by which we turn a Markov chain into a number of random walks, using first return probabilities. Section \ref{sec:groups} then presents background material on hyperbolic groups and spaces, geodesic combings, and random walks. 

These techniques are combined in Section \ref{sec:Markov_action} to show that sample paths in the geodesic combing of $G$ converge to the boundary of $X$ and have positive drift. This proves Theorem \ref{th:intro_1} and Theorem \ref{th:intro_2}. We also show that with probability going to $1$, the $n$th step of a sample path in the Markov chain is $X$--loxodromic with translation length growing linearly in $n$. One of the main steps in establishing these results is to show that the first return probabilities discussed above determine nonelementary measures on the group $G$.

Finally, in Section \ref{sec:counting}, we turn our statements about the Markov measure into statements about counting in the Cayley graph of the group $G$. This culminates in the proofs of Theorem \ref{th: intro_gen_drift} and Theorem \ref{th:intro_gen_trans}.

\subsection*{Acknowledgements}
The authors thank Ilya Kapovich and Joseph Maher for helpful suggestions.
The first author is partially supported by NSF grant DMS-1401875, and second author is partially supported by NSF grant DMS-1400498.

\section{Counting paths in graphs} \label{sec:counting_graphs}

We start by setting up our notation and recalling some fundamental facts about directed graphs. Note that we will first 
deal with general graphs, and introduce the group structure only later. We mostly follow Calegari--Fujiwara \cite{calegari2010combable}.

\subsection{Almost semisimple graphs}

Let $\Gamma$ be a finite, directed graph with vertex set $V(\Gamma) = \{v_0, v_1, \dots, v_{r-1}\}$. 
The \emph{adjacency matrix} of $\Gamma$ is the $r \times r$ matrix $M = (M_{ij})$ defined so that 
$M_{ij}$ is the number of edges from $v_i$ to $v_j$. 

Such a graph is \emph{almost semisimple} of growth $\lambda > 1$ if the following hold: 
\begin{enumerate}
\item 
There is an \emph{initial vertex}, which we denote as $v_0$; 
\item 
For any other vertex $v$, there is a (directed) path from $v_0$ to $v$; 
\item 
The largest modulus of the eigenvalues of $M$ is $\lambda$, and for any eigenvalue of modulus $\lambda$, its geometric 
multiplicity and algebraic multiplicity coincide.
\end{enumerate}

The \emph{universal cover} of $\Gamma$ is the countable tree $\widetilde{\Gamma}$ whose vertex set 
is the set of finite paths in $\Gamma$ starting from $v_0$, and there is an edge $p_1 \to p_2$ in $\widetilde{\Gamma}$ if 
the path $p_2$ is the concatenation of $p_1$ with an edge of $\Gamma$. 

The set $\widetilde{\Gamma}$ is naturally a rooted tree, where the root is the path of length $0$ from $v_0$ to itself, hence it can be identified with $v_0$.
Moreover, $\widetilde{\Gamma}$ is naturally endowed with a distance coming from the tree structure: there, a vertex has distance $n$ from the initial vertex 
if and only if it represents a path of length $n$. We denote as $| g |$ the distance between $g \in \widetilde{\Gamma}$ and the root.
There is a natural graph map $\widetilde{\Gamma} \to \Gamma$ which sends each vertex in $\widetilde{\Gamma}$ representing a path in $\Gamma$ 
to its endpoint, seen as a vertex of $\Gamma$.
For any vertex $g$ of $\widetilde{\Gamma}$, we denote as $[g]$ the corresponding vertex of $\Gamma$.

Given two vertices $v_1$, $v_2$ of a directed graph, we say that $v_2$ is \emph{accessible from} $v_1$ and write $v_1 \to v_2$ if there is a path from $v_1$ to $v_2$. 
Then we say that two vertices are \emph{mutually accessible} if $v_1 \to v_2$ and $v_2 \to v_1$. 
Mutually accessibility is an equivalence relation, and equivalence classes are called \emph{irreducible components} of $\Gamma$. 

The \emph{boundary} of $\widetilde{\Gamma}$ is the set $\partial \widetilde{\Gamma}$ of infinite paths in $\Gamma$ starting from the initial vertex. 
The set $\widetilde{\Gamma} \cup \partial \widetilde{\Gamma}$ carries a natural metric, where the distance between two (finite or infinite) paths 
$\gamma_1$ and $\gamma_2$ is $d(\gamma_1, \gamma_2) = 2^{-P(\gamma_1, \gamma_2)}$ where $P(\gamma_1, \gamma_2)$ is the length of the longest 
path from the initial vertex which is a common prefix of both $\gamma_1$ and $\gamma_2$. 
This way, $\widetilde{\Gamma} \cup \partial \widetilde{\Gamma}$ is a compact metric space, and $\partial \widetilde{\Gamma}$ is a Cantor set.

\subsection{Counting measures}
For each $n$, we denote as $S_n \subseteq \widetilde{\Gamma}$ the sphere of radius $n$ around the origin, i.e. the set of vertices at distance $n$ from the initial vertex.  This defines a sequence $P^n$ of \emph{counting measures} on $\widetilde{\Gamma}$: namely, for each set $A \subset \widetilde{\Gamma}$ 
we define
$$P^n(A):= \frac{\# (A \cap S_n)}{\# S_n}.$$
For any subset $A \subseteq \widetilde{\Gamma}$, we define the \emph{growth} $\lambda(A)$ of $A$ as 
$$\lambda(A) := \limsup_{n \to \infty} \sqrt[n]{\# (A \cap S_n)}.$$
By construction, the growth of $\widetilde{\Gamma}$ is $\lambda$ (see also Lemma \ref{L:growth}). 

For each vertex $v$ of $\Gamma$ which lies in a component $C$, let $\P_v(C)$ denote the set of finite paths in $\Gamma$ based at $v$ which lie entirely in $C$. 
Moreover, for any lift $g$ of $v$ to $\GammaT$, we let $\P_g(C)$ be the set of finite paths in $\GammaT$ based at $g$ whose projection to $\Gamma$ lies entirely in $C$. 

We call an irreducible component $C$ of $\Gamma$ \emph{maximal} if for some (equivalently, any) $g$ which projects to an element of $C$, the growth of $\P_g(C)$ equals $\lambda$.

\subsection{Vertices of small and large growth}

\begin{definition}
For each vertex $g \in \widetilde{\Gamma}$, 
the \emph{cone} of $g$, denoted as $cone(g)$, is the set of (finite or infinite) paths in $\widetilde{\Gamma}$ 
starting at $v_0$ and passing through $g$. 
\end{definition}

We now define vertices to be of large or small growth, according to the growth of their cone. 
More precisely, following Calegari-Fujiwara \cite{calegari2010combable}, we define the linear map $\rho:  \mathbb{R}^{V(\Gamma)} \to \mathbb{R}^{V(\Gamma)}$ as

$$\rho(v) := \lim_{N \to \infty} \frac{1}{N} \sum_{n=0}^{N} \frac{M^n v}{\lambda^n}$$

The limit exists by the almost semisimplicity assumption (3). 
If we denote by $1$ the vector in $\mathbb{R}^{V(\Gamma)}$ all of whose coordinates are $1$, note that by construction
$(M^n 1)_i$ equals the number of paths of length $n$ starting at $v_i$, i.e. for any $g \in \widetilde{\Gamma}$ we have
 $$(M^n 1)_i = \#( cone(g) \cap S_{n+|g|})$$
 where $[g] = v_i$. In particular, $(M^n 1)_0 = \#S_n$. 

\begin{definition}
We say a vertex $v_i$ of $\Gamma$ has \emph{small growth} if $\rho(1)_i = 0$, and has \emph{large growth} if $\rho(1)_i > 0$.
\end{definition}

Note that if $v_i$ is of small growth and there is an edge $v_i \to v_j$, then $v_j$ is also of small growth. We denote as $LG$ the set of 
vertices of large growth. Note that if a vertex belongs to a maximal irreducible component, then it has large growth, but the converse is not 
necessarily true; indeed, the cone of a vertex may have maximal growth, while its irreducible component may have smaller growth. 

\begin{lemma} \label{L:growth}
Let $\Gamma$ be an almost semisimple graph of growth $\lambda > 1$.
\begin{enumerate}
\item 
There exists $c >0$ such that for any vertex $v$ of large growth and any $n \geq 0$, 
$$c^{-1} \lambda^n \leq \# \{ \textup{paths from }v\textup{ of length }n \} \leq c \lambda^n$$
\item 
There exists $c > 0$ and $\lambda_1 < \lambda$ such that, for any vertex $v$ of small growth and any $n \geq 0$, 
$$\# \{ \textup{paths from }v\textup{ of length }n \} \leq c \lambda_1^n$$
\item 
Let $v$ be a vertex which belongs to a maximal component $C$. Then there exists $c > 0$ such that for any $n \geq 0$, 
$$c^{-1} \lambda^n \leq \# \{\textup{paths in } \mathcal{P}_v(C) \textup{ of length }n \} \leq c \lambda^n$$
\end{enumerate}
\end{lemma}

\begin{proof}
Let us first assume that $v = v_0$ is the initial vertex, which has large growth by construction. 
By writing $M$ in Jordan normal form, one has $\Vert M^n \Vert \asymp \lambda^n$ for any submultiplicative norm. 
Since all norms on the space of matrices are equivalent, we also have 
$$\sum_{i,j} (M^n)_{ij} \asymp \lambda^n$$
which can be rewritten as 
$$\sum_i (M^n 1)_i \asymp \lambda^n$$ 
Now, for each vertex $v_j$ there is a path of length $l_j$ from $v_0$ to $v_j$, hence 
$$(M^{n+l_j} 1)_0 \geq (M^n 1)_j$$
Moreover, since the degree of each vertex is bounded by some $d$, 
$$(M^{n+l_j} 1)_0 \leq d^{l_j} (M^n 1)_0$$
thus putting together the estimates yields 
$$(M^n 1)_0 \leq \sum_i (M^n 1)_i \leq \sum_i d^{l_i} (M^n 1)_0$$
which gives
$$\#S_n = (M^n 1)_0 \asymp \lambda^n$$
proving (1) in this case. 

Let us now pick a vertex $v$, and consider the set $\mathcal{V}$ of all vertices which are accessible from $v$. 
Then the matrix $M^T$ leaves invariant the span $V'$ of all basis vectors corresponding to elements of $\mathcal{V}$.
Let $M_1$ be the restriction of $M^T$ to $V'$; then either $M_1$ is almost semisimple of growth $\lambda$, or it has spectral radius $< \lambda$. 
In the first case, the above proof shows that 
$$\#(cone(v) \cap S_n) \asymp \lambda^n$$
which implies that $v$ has large growth; in the second case, then $(M_1^n 1)_v \leq \Vert M_1^n \Vert \leq c \lambda_1^n$ with $\lambda_1 < \lambda$, 
hence $v$ has small growth, proving (2).

To prove (3), let us consider the set $\mathcal{W}$ of vertices which belong to the maximal component $C$, let $V''$ be the span 
of its basis vectors. Now, $M_1^T$ leaves invariant $V''$, and let $M_2$ be the restriction of $M_1^T$ to $V''$. Again, either $M_2$ is almost 
semisimple or it has spectral radius strictly smaller than $\lambda$. However, since $C$ is maximal the spectral radius cannot be strictly smaller than 
$\lambda$, hence if $v = v_i$ we have 
$$ \# \{\textup{paths in } \mathcal{P}_v(C) \textup{ of length }n \}  = (M_2^n 1)_i  \asymp \lambda^n .$$  
\qedhere

\end{proof}

\begin{remark}
In order to clarify the connection between large growth vertices and maximal components, let us note that one can prove the following:
if $v$ is a vertex which belongs to a component $C$, then $v$ has large growth if and only if there exists a descendent component of $C$ (possibly $C$ itself) 
which is maximal.
\end{remark}

Let us denote as $S_n$ the set of elements of $\GammaT$ at distance $n$ from the initial vertex. If $g \in S_n$, then we denote $\widehat{g}$ the element along the path from the initial vertex to $g$ at distance $n-\log n$ from the initial vertex.

\begin{proposition} \label{P:smallg}
We have 
$$\frac{\#\{g \in S_n  \ : \ \widehat{g} \textup{ of small growth} \}}{\# S_n } \to 0$$
as $n \to \infty$.
\end{proposition}

\begin{proof}
By Lemma \ref{L:growth} (1), the number of paths from the origin of length $n - \log n$ is at most $c \lambda^{n - \log n}$, and 
by Lemma \ref{L:growth} (2) for each vertex of small growth the number of outgoing paths of length $\log n$ is at most $c \lambda_1^{\log n}$, 
so the total number of paths of length $n$ from the origin which passes at time $n - \log n$ through a vertex of small growth is at most  
$$c \lambda^{n - \log n} \cdot c \lambda_1^{\log n} \leq c^2 \lambda^n n^{\log \lambda_1 - \log \lambda}$$
which is negligible with respect to the cardinality of $S_n$, of the order of $\lambda^n$.
\end{proof}

\subsection{The Patterson-Sullivan measure}

We define the \emph{Patterson-Sullivan measure} $\nu$ as the limit of the measures
\[
\nu_N := \frac{\sum_{|g|  \leq N} \lambda^{-|g|} \delta_g}{\sum_{|g|  \leq N} \lambda^{-|g|}} 
\]
on $\widetilde{\Gamma} \cup \partial \widetilde{\Gamma}$.
As a consequence of the almost semisimplicity one gets the following proposition. For its statement, recall that $1$ denotes the vector with all entries equal to $1$.

\begin{proposition}
The limit $\nu := \lim_{N \to \infty} \nu_N$ exists as a limit of measures on $\widetilde{\Gamma} \cup \partial \widetilde{\Gamma}$, 
and it is given on cones by the formula: 
\begin{equation} \label{E:PS}
\nu(cone(g))  = \left\{ \begin{array}{ll} \frac{\rho(1)_i}{\rho(1)_0} \lambda^{-|g|} & \textup{ if } [g] = v_i \textup{ has large growth}\\
0 & \textup{ otherwise}
\end{array} \right.
\end{equation}
In particular, $\nu$ is supported on $\partial \widetilde{\Gamma}$. The measure $\nu$ is called the \emph{Patterson-Sullivan measure}.
\end{proposition}

For a proof, see the discussion preceding Lemma 4.19 of \cite{calegari2010combable}. The only difference is that they
consider the limit of the measures $\widehat{\nu}_N := \frac{1}{N} \sum_{|g|  \leq N} \lambda^{-|g|} \delta_g$, so the limit may not be a probability measure.
We normalize by dividing by $\rho(1)_0$ so that $\nu$ is a probability measure.



\section{Markov chains} \label{sec:Mchains}

Let $\Gamma$ be a finite directed graph, with vertex set $V(\Gamma)$ and edge set $E(\Gamma)$.

A \emph{transition probability} $\mu$ for $\Gamma$ is a function $\mu : E(\Gamma) \to \mathbb{R}^{\geq 0}$
such that for each vertex $v \in V(\Gamma)$, the total sum of the probabilities of the edges going out from $v$ is $1$: 
$$\sum_{e \in Out(v)} \mu(e) = 1$$
where $Out(v)$ denotes the set of outgoing edges from vertex $v$.
A \emph{vertex distribution} for $\Gamma$ is a function $p: V(\Gamma) \to \mathbb{R}^{\geq 0}$ such that 
$\sum_{v \in V(\Gamma)} p(v) = 1$.

Let $\Omega$ denote the set of all infinite paths in $\Gamma$, no matter what the starting vertex is.
Given a transition probability $\mu$ and a vertex distribution $p$, we define a probability measure $\mathbb{P}_{p, \mu}$ 
on $\Omega$  by assigning to each cylinder set 
$$C_{e_1, \dots, e_n} = \{ \omega = (\omega_n) \in \Omega  \ : \ \omega_1 = e_1, \dots, \omega_n = e_n  \}$$
the probability 
$$\mathbb{P}_{p, \mu}\left(C_{e_1, \dots, e_n}\right)  = p(v_0) \mu(e_1) \cdots \mu(e_n)$$
where $v_0$ is the initial vertex of the edge $e_1$. Finally, if $\gamma = e_1 \dots e_n$ is a finite path in $\Gamma$ or $\widetilde{\Gamma}$, 
we define
$$\mu(\gamma)  = \mu(e_1)\dots \mu(e_n)$$
to be the product of the transition probabilities of its edges. 

\subsection{Recurrence}

Let us now fix a transition probability $\mu$ on the graph $\Gamma$.

Then for each vertex $v$, we denote as $\mathbb{P}_v = \mathbb{P}_{\delta_v, \mu}$ 
the probability measure on $\Omega$ where the vertex distribution is the $\delta$-measure at $v$. Clearly, this measure 
is supported on the set of infinite paths starting at $v$, which we denote by $\Omega_v$. 
Finally, we denote simply as $\mathbb{P} = \mathbb{P}_{v_0}$ the probability measure on the set of infinite paths starting at the initial vertex $v_0$. Set $\Omega_0 = \Omega_{v_0}$.

The measure $\mathbb{P}$ defines a Markov process on $\widetilde{\Gamma}$, which starts at the initial vertex $v_0$
and moves along the graph according to the transition probabilities. 
To be precise, let us define the map $g_n : \Omega_{0} \to \widetilde{\Gamma}$ which associates to each infinite path starting at the initial vertex
its prefix of length $n$, seen as a vertex of $\widetilde{\Gamma}$.
This defines a sequence of \emph{Markov measures} $(\mathbb{P}^n)$ on $\widetilde{\Gamma}$ by 
setting $\mathbb{P}^n(A)$ as the probability that the Markov chain hits $A$ at time $n$, that is $\mathbb{P}^n(A) = \mathbb{P}(g_n \in A)$.


\begin{definition}
A vertex $v$ of $\Gamma$ is \emph{recurrent} if: 
\begin{enumerate}
\item there is a path from $v_0$ to $v$ of positive probability; and 
\item
whenever there is a path from $v$ to another vertex $w$ of positive probability, there is also a path from $w$ to $v$ 
of positive probability. 
\end{enumerate}

Two recurrent vertices $v_1$ and $v_2$  are \emph{equivalent} if there exists a path from $v_1$ to $v_2$ of positive probability, 
and a path from $v_2$ to $v_1$ of positive probability. 
A \emph{recurrent component} is an equivalence class of recurrent vertices. 
\end{definition}

For each $n$, let us define $X_n : \Omega \to \Gamma$ so that $X_n(\omega)$ denotes the endpoint of the finite prefix 
of $\omega$ of length $n$. 
Once we fix a measure on $\Omega$, the sequence $(X_n)$ is a stochastic process with values in $\Gamma$ (in our case, a Markov chain) 
and $X_n$ is the location of the nth step of the chain. 

\begin{lemma}\label{L:recurrent}  \label{lem:first_return}
Let $\Gamma$ be a directed graph with transition probability $\mu$.
\begin{enumerate}
\item
For $\mathbb{P}$-almost every path in $\Omega$, there exists a recurrent component $C$ and an index $N$ such that 
$X_n$ belongs to $C$ for all $n \geq N$.
Moreover, $X_n$ visits every vertex of $C$ infinitely many times.
\item
There exists $c > 0$ such that, for any recurrent vertex $v$, for each $n\ge0$ we have
$$\mathbb{P}_v\left( \tau^+_v = n \right) \leq e^{-cn}$$
where $\tau^+_v = \min \{ n \geq 1 \ : \ X_n = v \}$ denotes the first return time to vertex $v$. 
\end{enumerate}
\end{lemma}

\begin{proof}
For each pair of vertices $v_i, v_j$ such that there is a path from $v_i$ to $v_j$ of positive probability, 
let us pick one such path $\gamma_{ij}$, and let $\mathcal{L}$ be the (finite) collection of all $\gamma_{ij}$. 
Now, let $r$ be the maximum length of all the elements of $\mathcal{L}$, and  $p := \min_{\gamma \in \mathcal{L}} \mu(\gamma) > 0$. 
For each non-recurrent vertex $v_i$ which can be reached from $v_0$ with positive probability, 
there is a path $\gamma$ in $\mathcal{L}$ from $v_i$ to some recurrent $v_j$. 
Then for each non-recurrent $v_i$ we have  
$$\mathbb{P}_{v_i}(X_n \textup{ not recurrent for all }n \leq r) \leq 1-p$$
Thus, for any $k$ we have
$$\mathbb{P}_{v_i}(X_n \textup{ not recurrent for all }n \leq k r) \leq (1-p)^k$$
yielding the first claim if $k \to \infty$.
Now, if $v_i$ and $v_j$ lie in the same recurrent component, then there are paths in $\mathcal{L}$ from $v_i$ to $v_j$ and from $v_j$ to $v_i$, 
hence
$$\mathbb{P}_{v_i}(X_n \neq v_j \textup{ for all }n \leq r) \leq 1-p$$
and by repeated application of this argument (as you never leave a recurrent component)
$$\mathbb{P}_{v_i}(X_n \neq v_j \textup{ for all }n \leq k r) \leq (1-p)^k$$
which implies the second claim. The claim about the return time follows from the above inequality setting $v_i = v_j = v$ 
a recurrent vertex. 
\end{proof}

\subsection{Markov measure} \label{sec:markov}
On the other hand, given an almost semisimple graph $\Gamma$ of growth $\lambda$, we construct a Markov chain by defining a transition probability  $\mu \colon E(\Gamma) \to \mathbb{R}^{\ge0}$ as follows: 
if $v_i$ has large growth, then for each directed edge $e$ from $v_i$ to $v_j$ we define 
\begin{equation} \label{E:Markov}
\mu(e) := \frac{ \rho(1)_j}{\lambda \rho(1)_i}.
\end{equation}
If $v_i$ has small growth, we define $\mu(e) =0$ for each directed edge from $v_i$ to $v_j$ if $i\ne j$ and $\mu(e) = 1/d_i$ if $e$ is one of $d_i$ directed edges from $v_i$ to itself. Let $\mu(v_i \to v_j)$ denote the probability of going from $v_i$ to $v_j$ and observe that if $v_i$ has large growth, then
\[
\mu(v_i \to v_j) = \frac{M_{ij} \rho(1)_j}{\lambda \rho(1)_i},
\]
and if $v_i$ has small growth, $\mu(v_i \to v_j) =0$ for $i\neq j$ and $\mu(v_i \to v_i) = 1$.
With this observation, one can easily check that $\mu$ defines a transition probability for $\Gamma$ as in \cite[Lemma 4.9]{calegari2010combable}.


For this choice of transition probability $\mu$, it is now easy to explicitly compute the $n$-step distributions $\mathbb{P}^n$ of 
the associated Markov chain. In fact, from eq. \eqref{E:Markov} one gets for each vertex $g \in \widetilde{\Gamma}$  
\begin{equation} \label{E:Markov-cone}
\mathbb{P}^n(\{g \}) = \left\{ \begin{array}{ll} \frac{\rho(1)_i}{ \rho(1)_0} \lambda^{-n} & \textup{ if } |g| = n, [g] = v_i \textup{ has large growth}\\
0 & \textup{ otherwise}
\end{array} \right.
\end{equation}

One can check from direct calculation, comparing formulas \eqref{E:PS} and \eqref{E:Markov-cone} that the sequence $(\mathbb{P}^n)$ of measures on $\widetilde{\Gamma} \cup \partial \widetilde{\Gamma}$
converges to the Patterson-Sullivan measure $\nu$. In particular, if we identify $\Omega_0$ with $\partial \widetilde \Gamma$ via path lifting to the base vertex, then $\nu = \mathbb{P}$.

\begin{lemma}
Let $\Gamma$ be an almost semisimple graph 
 with transition probability $\mu$ given by eq. \eqref{E:Markov}. Then an irreducible component of $\Gamma$ is maximal if and only if it is recurrent.
\end{lemma}

\begin{proof}
Let $C$ be an irreducible component of $\Gamma$, and $v_i \in C$ a vertex. If $C$ is recurrent, then for each $n$ almost every path 
of length $n$ which starts from $v_i$ stays in $C$, hence by equation \eqref{E:Markov} one gets
$$1 = \sum_{v_j \in C} \mathbb{P}_{v_i}(X_n = v_j) = \sum_{j \in C} \frac{M^n_{ij} \rho(1)_j}{\lambda^n \rho(1)_i} \leq \left( \frac{\max_{v_j \in C} \rho(1)_j}{\rho(1)_i} \right)
\frac{ \sum_{v_j \in C} M^n_{ij} }{\lambda^n }$$
Note that $\sum_{j \in C} M^n_{ij}$ equals the number of paths of length $n$ which start from $v_i$ and lie entirely in $C$, hence by the above estimate 
such number is bounded below by $c \lambda^n$, proving that $C$ is maximal.
Conversely, suppose that $C$ is maximal. Then each vertex of $C$ has large growth and so for each $v$ in $C$ there is a path from $v_0$ to $v$ of positive probability. If there is a path from $v$ to some vertex $w$ of positive probability, then $w$ also has large growth (eq. \eqref{E:Markov}). Hence, $w \in C$ since $\Gamma$ is almost semisimple. Now any directed path from $w$ back to $v$ in $C$ has positive probability by definition of $\mu$, showing that $C$ is recurrent and completing the proof.
\end{proof}

Finally, the next lemma compares the $n$-step distribution for the Markov chain to the counting measure. 

\begin{lemma} \label{lem:comparePn}
There exists $c > 1$ such that, for each $A \subseteq \widetilde{\Gamma}$, we have the inequality
\begin{equation} \label{E:countmu}
c^{-1} \ \mathbb{P}^n(A) \leq P^n(A \cap LG) \leq c \ \mathbb{P}^n(A).
\end{equation}
\end{lemma}

\begin{proof}
By definition, 
$$P^n(A \cap LG) = \frac{\# (A \cap LG \cap S_n)}{ \# S_n}$$
while from formula \eqref{E:Markov-cone} 
$$\mathbb{P}^n(A) = \left( \frac{1}{\rho(1)_0} \right) \sum_{g \in A \cap LG \cap S_n}  \rho(1)_g \lambda^{-n}$$
Now, since there are a finite number of vertices, we have $c^{-1} \leq \rho(1)_g \leq c$ for any $g$ of large growth, hence 
$$\mathbb{P}^n(A) = \left( \frac{1}{\rho(1)_0} \right) \sum_{g \in A \cap LG \cap S_n}  \rho(1)_g \lambda^{-n} \asymp \#(A \cap LG \cap S_n) \lambda^{-n}$$
and since $\#S_n \asymp \lambda^n$ we get the claim.
\end{proof}

\section{From Markov chains to random walks} \label{sec:markov_to_walk}

Let $\Gamma$ be the finite graph, $v_0$ its initial vertex, and $\mathbb{P}$ the Markov measure on the space $\Omega_0$ of 
infinite paths starting at the initial vertex. 

We will associate to each recurrent vertex of the Markov chain a random walk, and use previous results on random walks to 
prove statements about the asymptotic behavior of the Markov chain.

\subsection{Return times}
For each sample path $\omega \in \Omega_0$, let us define $n(k,v, \omega)$ as the $k^{th}$ time the path $\omega$ lies at the vertex $v$. 
In formulas, 
$$n(k, v, \omega) := \left\{ \begin{array}{ll} 0 & \textup{if } k = 0 \\ \min \left( h > n(k-1, v, \omega) \ : \ X_h(\omega) = v \right) &  \textup{if }k \geq 1 \end{array} \right.$$
To simplify notation, we will write $n(k,v)$ instead of $n(k,v, \omega)$ when the sample path $\omega$ is fixed.  Moreover, we denote 
$$\tau(k,v) := n(k,v) - n(k-1, v).$$
These are the return times to the vertex $v$.

\begin{definition}
For every vertex $v$ of $\Gamma$, the \emph{loop semigroup} of $v$ is the set  $L_v$ of loops in the graph 
$\Gamma$ which begin and end at $v$. It is a semigroup with respect to concatenation.
A loop in $L_v$ is \emph{primitive} if it is not the concatenation of two (non-trivial) loops in $L_v$. 
\end{definition}

Let us now pick a recurrent component $C$, and let $\Omega_C$ be the set of all infinite paths from the initial vertex 
which enter $C$ and remain inside $C$ forever. We denote as $\mathbb{P}_C$ the conditional probability of $\mathbb{P}$ on $\Omega_C$.

Let us now fix a vertex $v$ of $C$. Then $\mathbb{P}_C$-almost every infinite path $\gamma$ in $\Omega_C$ passes through $v$ 
infinitely many times, hence such $\gamma$ can be decomposed as a concatenation
$\gamma = \gamma_0 \gamma_1 \dots \gamma_n \dots$ where $\gamma_0$ is a path from $v_0$ to $v$ which visits $v$ only once, and each $\gamma_i$ 
for $i \geq 1$ is a primitive loop at $v$. This defines for each recurrent vertex $v \in C$ a measurable map 
$$\varphi_v : (\Omega_C, \mathbb{P}_C) \to (L_v)^\mathbb{N}$$
$$(e_n) \to (\gamma_{1} , \gamma_{2}, \ldots)$$
where $\gamma_{k} =  e_{n(k,v)+1} \dots e_{n(k+1,v)}$. 

We now define the \emph{first return measure} $\mu_v$ on the set 
of primitive loops by setting, for each primitive loop $\gamma = e_1 \dots e_n$ 
with edges $e_1, \dots, e_n$, 
$$\mu_v(e_1 \dots e_n) = \mu(e_1)\dots \mu(e_n).$$
Extend $\mu_v$ to the entire loop semigroup $L_v$ by setting $\mu_v(\gamma) = 0$ if $\gamma\in L_v$ is not primitive. 
Since almost every path starting at $v$ visits $v$ infinitely many times, the measure $\mu_v$ is a probability measure.

\begin{lemma} \label{lem:getting_iids} For any recurrent vertex $v$ which belongs to a component $C$, we have
$$(\varphi_v)_\star \mathbb{P}_C = \mu_v^\mathbb{N}.$$
\end{lemma}

\begin{proof}
It suffices to show that the measures $\mu_v^\mathbb{N}$ and $(\varphi_v)_\star \mathbb{P}_C$ agree on cylinder sets of the form $C_{l_1,l_2, \ldots, l_k} \subset (L_v)^\mathbb{N}$. We may also assume that each $l_i$ is primitive, otherwise $C_{l_1,l_2, \ldots, l_k}$ has $\mu_v^\mathbb{N}$--measure $0$ and $\varphi_v^{-1}(C_{l_1,l_2, \ldots, l_k})$ is empty. Hence, $\mu_v^\mathbb{N} (C_{l_1,l_2, \ldots, l_k}) = \mu_v(l_1)\ldots \mu_v(l_k)$.

Let $H_v$ be the set of paths in $\Omega_C$ which start at $v_0$ and meet $v$ only at their terminal endpoint. Then we have the disjoint union
\[
\varphi_v^{-1}(C_{l_1,l_2, \ldots, l_k}) = \bigcup_{\gamma \in H_v} \gamma \cdot l_1 l_2 \ldots l_k \cdot \Omega_v.
\]

So by the Markov property
\begin{align*}
\mathbb{P}_C(\varphi_v^{-1}(C_{l_1,l_2, \ldots, l_k})) &= \sum_{\gamma \in H_v} \mathbb{P}_C(\gamma) \mathbb{P}_v(l_1) \ldots \mathbb{P}_v(l_k) \\
&= \mu_v(l_1)\ldots \mu_v(l_k),
\end{align*}
where the second equality follows from the fact that almost every path in $\Omega_C$ hits $v$ and that each $l_i$ is a primitive loop. This completes the proof.
\end{proof}

By Lemma \ref{lem:first_return}, for every recurrent vertex $v$, the first return measure $\mu_v$ has finite exponential moment, 
i.e. there exists a constant $\alpha >0$ such that 
\begin{equation} \label{e:expo}
\int_{L_v} e^{\alpha|\gamma|} \ d\mu_v(\gamma) < \infty.
\end{equation}

\section{Groups, combings and random walks} \label{sec:groups}
The connection between hyperbolic groups and Markov chains is through the theory of geodesic combings.
\begin{definition}
A \emph{finite state automaton} over the alphabet $S$ is a finite, directed graph $\Gamma$, with a distinguished \emph{initial vertex} $v_0$
and such that its edges are labeled by elements of $S$, in such a way that no vertex has two outgoing edges with the same label.
\end{definition}
 
Let $G$ be a countable group, and $S \subset G$ a finite set which generates $G$ as a semigroup. Then we say that $G$ has a \emph{geodesic combing} with respect to $S$
if there exists a finite state automaton $\Gamma$ such that: 
\begin{enumerate}
\item every path in $\Gamma$ starting at the initial vertex is a geodesic in the word metric for $S$; 
\item for every element $g \in G$, there exists exactly one path in $\Gamma$ which starts at the initial vertex and represents $g$.
\end{enumerate}
In this paper, we will only need geodesic combings on hyperbolic groups. For detailed introduction, see \cite{calegari2013ergodic}. We begin by reviewing some basics on hyperbolic metric spaces.

\subsection{Hyperbolic spaces and groups} \label{sec:hyperbolic}
A geodesic metric space $X$ is \emph{$\delta$--hyperbolic}, where $\delta\ge0$, if for every geodesic triangle in $X$, each side is contained within the $\delta$--neighborhood of the other two sides. Each hyperbolic space has a well-defined \emph{Gromov boundary} $\partial X$, and we refer the reader to \cite[Section III.H.3]{BH}, \cite{GhysdelaHarpe}, or \cite[Section 2]{kap_boundaries} for definitions and properties. We note that when the metric space $X$ is proper (meaning that closed metric balls are compact), $X\cup \partial X$ is a compactification of $X$.
If a group $G$ acts on $X$ by isometries, then there is an induced action $G \curvearrowright \partial X$ by homeomorphisms.

Recall that we have defined a \emph{nonelementary} action of $G$ on $X$ as an action $G \curvearrowright X$ for which there are $g,h \in G$ which act loxodromically on $X$ and whose fixed point sets on the Gromov boundary of $X$ are disjoint. We say that two such elements of $G$ are independent loxodromics for the action $G \curvearrowright X$.
This definition extends equally well to an action of a semigroup $L$ on $X$. We will require the following criterion which characterizes when a semigroup action $L \curvearrowright X$ is nonelementary:

\begin{proposition}[{\cite[Proposition 7.3.1]{das2014geometry}}] \label{th:Tushar}
Suppose that $L$ is a semigroup which acts on the hyperbolic metric space $X$ by isometries. If the limit set $\Lambda_L \subset \partial X$ of $L$ on the boundary of X is nonempty and $L$ does not have a finite orbit in $\partial X$, then 
the action $L \curvearrowright X$ is nonelementary \footnote{Das--Simmons--Urbanski call such an action \emph{of general type}.}, that is $L$ contains two independent loxodromics.
\end{proposition}

If $G$ is a finitely generated group, then $G$ is \emph{hyperbolic} if for some (any) finite generating set $S$, the associated Cayley graph $C_S(G)$ is $\delta$--hyperbolic for some $\delta\ge0$. Since $C_S(G)$ is locally finite, the Gromov boundary $\partial G = \partial C_S(G)$ is compact. In fact, the boundary of $G$ can be identified with the set of asymptote classes of geodesic rays starting at the identity, where two geodesic 
rays $\gamma_1$, $\gamma_2 : [0, \infty) \to G$ are asymptotic if $\sup_{t} d(\gamma_1(t), \gamma_2(t)) < \infty$.
Since the boundary $\partial G$ is compact and Hausdorff, it is a Baire space -- the union of closed nowhere dense sets has empty interior.\\

Hyperbolic groups have automatics structures by the well-known theorem of Cannon:

\begin{theorem}[\cite{cannon1984combinatorial}] \label{th:combing}
Let $G$ be a hyperbolic group, and $S$ a finite generating set. Then $G$ has a geodesic combing with respect to $S$. 
\end{theorem}
 
Let $\Gamma$ be the directed graph associated to the geodesic combing of $G$, whose edges are labeled by generators from $S$. As $\Gamma$ is a directed graph, we will use the notation and terminology established in Section \ref{sec:counting_graphs}. That $\Gamma$ comes from a combing of $G$ implies that there is a map $\ev \colon \widetilde{\Gamma} \to G$, the \emph{evaluation map}, which associates to each directed path $p$ in $\Gamma$ beginning at $v_0$, the element $\ev(p) \in G$ represented by the word obtained by concatenating the edge labels of $p$. By definition, $|\ev(p)| = |p|$ and $\ev \colon \widetilde \Gamma \to G$ is bijective. We remark that this map extends to \emph{all} directed paths in $\Gamma$ (no matter which vertex they start from), again by concatenating edge labels, and we will continue to denote this map by $\ev$. Note that the extension still maps directed paths to geodesics, but is no longer injective since paths which start at different vertices at $\Gamma$ may still read the same edge labels.
Moreover, since paths in $\widetilde{\Gamma}$ beginning at $v_0$ map to geodesics in $G$ beginning at $1$, $\ev$ induces a boundary map, denoted $\ev \colon \partial \widetilde{\Gamma} \to \partial G$, which is continuous, surjective, closed, and finite-to-one \cite{calegari2013ergodic}. 

In fact, the graph $\Gamma$ is almost semisimple \cite[Lemma 4.15]{calegari2010combable}, this follows from a theorem of Coornaert.

\begin{theorem}[{\cite[Theorem 7.2]{coornaert1993mesures}}] \label{th:Coorn}
Let $G$ be a nonelementary hyperbolic group with generating set $S$. Let $S_n$ be the set of elements of $G$ of word length $n$. Then there are constants $\lambda >0$, $c\ge1$ such that
\[
c^{-1}\lambda^n \le \#S_n \le c \lambda^n.
\]
\end{theorem}
\noindent We remark that since $\Gamma$ parameterizes the geodesic combing of $G$, the constant $\lambda >0$ in Theorem \ref{th:Coorn} is the growth rate of $\Gamma$.

Finally, the map $\ev \colon \partial \widetilde \Gamma \to \partial G$ pushes forward the Patterson-Sullivan measure on $\partial \widetilde{\Gamma}$ to the usual \emph{Patterson--Sullivan measure} on $\partial G$, which we continue to denote by $\nu$. The measure $\nu$ is a quasiconformal measure on $\partial G$, which can alternatively be defined, up to bounded ratios, as a weak limit of uniform measure on the balls $B_n$.
For details, see \cite{coornaert1993mesures} or Section 4.3 of \cite{calegari2010combable}. 

\begin{lemma}[\cite{coornaert1993mesures}\cite{calegari2013ergodic}] \label{L:ergodic}
The Patterson-Sullivan measure on $\partial G$ has full support, i.e. every open set has positive measure. 
Moreover, the action of $G$ on $\partial G$ is ergodic, i.e. any $G$-invariant subset of $\partial G$ 
has either zero or full measure.
\end{lemma}

Finally, we turn to the definition and basic properties of shadows in the $\delta$-hyperbolic space $X$. For $x,y \in X$, the \emph{shadow in $X$ around $y$ based at $x$} is 
\[
S_x(y,R) = \{z\in X : (y,z)_x \ge d(x,z) -R\},
\]
where $R> 0$ and $(y,z)_x = \frac{1}{2}(d(x,y) +d(x,z) -d(y,z))$ is the usual \emph{Gromov product}. The \emph{distance parameter} of $S_x(y,R)$ is by definition the number $r = d(x,y) -R$, which up to an additive constant depending only on $\delta$, measures the distance from $x$ to $S_x(y,R)$. Indeed, $z \in S_x(y,R)$ if and only if any geodesic $[x,z]$  $2\delta$--fellow travels any geodesic $[x,y]$ for distance $r+O(\delta)$. 

We will use shadows in both the hyperbolic group $G$ and the hyperbolic space $X$, but the location of the shadow will always be clear from context. For example, $S_1(g,R)$ will usually denote a shadow in $G$ based at the identity element.  \\

The following lemma is an easy exercise in hyperbolic geometry.
\begin{lemma} \label{lem:neigh-shadow}
For each $D \geq 0$, and each $x, y$ in a $\delta$-hyperbolic space, we have 
$$N_D(S_x(y, R)) \subseteq S_x(y, R + D+2\delta).$$
\end{lemma}

Returning to the directed graph $\Gamma$ parameterizing the geodesic combing for $G$,
we say a vertex $v$ is \emph{continuable} if there exists at least one infinite path starting from $v$. 
Otherwise, $v$ is said to be $\emph{non-continuable}$. Elements of $g$ are called continuable or non-continuable 
according to the property of their vertex. Note that $g$ is continuable if and only if it lies on an infinite geodesic ray in the combing. 

\begin{lemma} \label{l:continuable}
Let $r$ be the number of vertices in $\Gamma$. Then for every $g \in G$, there exists $h \in G$ with $d_G(g,h) \leq r$ and such that the vertex of $h$ is continuable. 
\end{lemma}

\begin{proof}
Note that every path of length $\geq r$ in the graph passes through some vertex at least twice, 
and a vertex which supports a loop is continuable. 
If $|g| \leq r$, just pick $h = 1$; otherwise, consider the path in the combing from $1$ to $g$, 
and look at the final subpath of length $r$ (the one ending in $g$). By the previous observation, such subpath has to hit a continuable vertex, proving the claim. 
\end{proof}

\begin{lemma} \label{l:shadow-large}
There exists a constant $c\ge 0$ which depends only on $\delta$ such that, 
for any $g$ continuable, and any $R \geq 0$,  the shadow $S_1(g, R+c)$ in $G$ has maximal growth.
\end{lemma}

\begin{proof}
Let $\gamma$ be an infinite geodesic in $G$ which contains $g$, and let $D = |g|$.
Moreover, let $g_n$ be a point on $\gamma$ at distance $n + D + c$ from $1$, where $c$ is a constant which depends on $\delta$.  
We claim that the ball of radius $n$ around $g_{n}$ is contained in $S_1(g, R)$. 
Indeed, if $h \in B_n(g_{n})$, then 
$$(h, g_n)_1 \geq d_G(1, g_n) - d_G(h, g_n) \geq n + D + c - n = D + c$$
Thus, a closest point projection $p$ of $h$ to $[1, g_n]$ lies on the segment $[g, g_n]$. 
As a consequence, $d(1,h) = d(1,g) + d(g, p) + d(p,h) + O(\delta)$ and $d(g, h) = d(g, p) + d(p, h) + O(\delta)$, hence
$$(h, g)_1 = d_G(1, g) + O(\delta)$$
which proves the claim. 
\end{proof}

\subsection{Random walks}
If $X$ is a separable hyperbolic space and $G$ acts on $X$ by isometries, then a probability measure $\mu$ on $G$ is said to be \emph{nonelementary} with respect to the action $G \curvearrowright X$ if the semigroup generated by the support of $\mu$ contains two independent loxodromics.

We will need the fact that a random walk on $G$ whose increments are distributed according to a nonelementary measure $\mu$ almost surely converge to the boundary of $X$ and has positive drift in $X$.

\begin{theorem}[{\cite[Theorems 1.1, 1.2]{MaherTiozzo}}] \label{thm:MT_drift}
Let $G$ be a countable group which acts by isometries on a separable hyperbolic space $X$, and let $\mu$ be a nonelementary probability distribution on $G$. Fix $x\in X$. Then,
almost every sample path $(w_nx)$ converges to a point in the boundary of $\partial X$, and the resulting hitting measure $\nu$ is nonatomic.

Moreover, if $\mu$ has finite first moment, then there is a constant $L>0$ such that for almost every sample path
$$\lim_{n \to \infty}\frac{d(x,w_nx)}{n} = L >0.$$
\end{theorem}
\noindent The constant $L >0$ in Theorem \ref{thm:MT_drift} is called the \emph{drift} of the random walk $(w_n)$.

\section{The Markov process and the action $G \curvearrowright X$}
\label{sec:Markov_action}
Let $G$ be a hyperbolic group with a nonelementary action $G \curvearrowright X$ on a hyperbolic space $X$. In this section, we relate the geometry of generic geodesics in $G$, with respect to Patterson--Sullivan measure, to the geometry of $X$ under the orbit map $G \to X$. These results will be used in the next section to prove genericity statements in $G$ with respect to counting in the Cayley graph. \\

We begin by fixing the notation needed throughout the remainder of the paper. Let $G$ be a hyperbolic group with finite generating set $S$ and let $\Gamma$ be the directed graph parameterizing the geodesic combing given by Theorem \ref{th:combing}. The directed edges of $\Gamma$ are labeled by $s \in S$ and we denote by $\Omega$ the set of all infinite paths in $\Gamma$ starting at any vertex $v \in \Gamma$. 

As $\Gamma$ is almost semisimple, we apply the machinery of Section \ref{sec:markov} to make the directed graph $\Gamma$ into a Markov chain with transition probability $\mu$ (see Equation \ref{E:Markov}).  This induces the measures $\mathbb{P}_v$ on the set $\Omega$,
one for each vertex $v \in \Gamma$ of large growth. As directed paths in $\Gamma$ evaluate to geodesics in $G$, the measure $\mathbb{P} = \mathbb{P}_{v_0}$, for $v_0$ the vertex associated to $1 \in G$, is thought of as a measure on the (combing) geodesic rays in $G$ starting at $1$. The other measures $\mathbb{P}_v$ have similar interpretations. Let 
$$
w_n \colon \Omega \to G
$$
 be the variable which associates to each path $p$ in $\Omega$ the element of $G$ spelled by concatenating the first $n$ edges of $p$. That is, $\mathbb{P}_v(w_n =g)$ is the probability of hitting $g\in G$ after $n$ steps starting at $v \in \Gamma$.


We will apply the notation and terminology developed in Sections \ref{sec:counting_graphs}, \ref{sec:Mchains}, and \ref{sec:markov_to_walk} to the Markov chain $\Gamma$. For example, for a subset $A \subset G$, $P^n(A)$ is the proportion of elements in the sphere $S_n$ of radius $n$ about $1 \in G$ which are contained in $A$. We also fix $\lambda >0$ to be the growth rate of $G$, as in Theorem \ref{th:Coorn}, and note that $\lambda$ is exactly the growth of $\Gamma$, as defined in Section \ref{sec:counting_graphs}. Finally, for $g \in G$, we set 
\[
\cone(g) = \ev(cone(\gamma)) \subset G \cup \partial G,
\]
where $\gamma \in \widetilde{\Gamma}$ with $\ev(\gamma) =g$. Informally, $\cone(g)\subset G \cup \partial G$ is the set of points which are reachable by geodesics of $G$ (parameterized by $\Gamma$) which pass through $g$.

\subsection{Patterson--Sullivan measure as a combination of harmonic measures} 
\label{sec: PS_comb}
We begin by realizing the Patterson--Sullivan measure on $\partial G$ as a combination of harmonic measures associated to random walks.

In Section \ref{sec:markov_to_walk}, to each vertex $v$ in $\Gamma$ we associated a loop semigroup $L_v$. Under evaluation, these semigroups map to subsemigroups of $G$, which we denote by $\Gamma_v$. In other words, $\Gamma_v$ is the semigroup of elements of $G$ which can be spelled by directed loops at the vertex $v \in \Gamma$. Hence, the first return measure $\mu_v$ on $L_v$ induces a probability measure, still denoted $\mu_v$, on $\Gamma_v \le G$, which has finite exponential moment (see eq. \eqref{e:expo}).

Let $\mathcal{R}$ be the set 
of recurrent vertices of $\Gamma$ and let
$\nu_v$ be the harmonic measure on $\partial G$ corresponding to the random walk $\mu_v$ on $G$. Let $N_v$ be the set of finite directed paths in $\Gamma$ beginning at $v_0$ and ending at $v$ which do not meet a recurrent vertex of $\Gamma$ in their interior. 

\begin{proposition} \label{prop:PS_comb}
The Patterson--Sullivan measure $\nu$ on $\partial G$ is a combination of the harmonic measures $\nu_v$:
\[
\nu = \sum_{v \in \mathcal{R}} \sum_{\gamma \in N_v} \mu(\gamma) \gamma_* \nu_v.
\]
\end{proposition}

\noindent In the statement of Proposition \ref{prop:PS_comb}, we have identified $\gamma \in N_v$ with its evaluation $\ev(\gamma) \in G$.

\begin{proof}
Let $\Pi \subset \Omega_0$ be the set of infinite paths which enter some maximal component and remain there forever. 
Let $\gamma \in \Pi$, and $v$ be the first recurrent vertex which is met along $\gamma$. 
Then $\gamma$ can be decomposed uniquely as $\gamma = \gamma_0 \gamma_1 \dots$ 
where $\gamma_0$ is a path from $v_0$ to $v$ which does not meet any recurrent vertex in its interior, 
and each $\gamma_i$ for $i \geq 1$ is a primitive loop at $v$.
Thus one can define the map $\varphi : \Pi \to \GammaT \times \bigcup_v (\Gamma_v)^\mathbb{N}$ as $\varphi(\gamma) = 
(\gamma_0, \gamma_1, \dots, \gamma_n, \dots)$. Since the set $\Pi$ has full $\mathbb{P}$-measure, and the 
Patterson-Sullivan measure $\nu$ on $\partial \GammaT$ is the same as $\mathbb{P}$ under the identification $\partial \GammaT = \Omega_{0}$, 
the map $\varphi$ can be defined as a measurable, $\nu$-almost surely defined map $\varphi : \partial \GammaT \to \GammaT \times \bigcup_v (\Gamma_v)^\mathbb{N}$.

Moreover, one has the boundary map $\textup{bnd} : G \times G^\mathbb{N} \to \partial G$ defined as  
$$\textup{bnd}(g_0, (g_1, \dots, g_n, \dots)) = \lim_{n \to \infty} g_0 g_1 \dots g_n$$
(whenever the limit exists). 

Combining these maps with the evaluation map, we get the following commutative diagram of measurable maps: 
$$\xymatrix{
\partial \GammaT \ar^{\ev}[r] \ar^\varphi[d] & \partial G \\ 
\GammaT \times \bigcup_v (\Gamma_v)^\mathbb{N} \ar^{\ev}[r] & G \times G^\mathbb{N} \ar^{\textup{bnd}}[u]  
}$$

The claim follows by pushing forward the Patterson-Sullivan measure $\nu$ on $\partial \GammaT$ along the diagram. 
First, by disintegration we get 
$$\nu = \sum_{v \in \mathcal{R}} \sum_{\gamma \in N_v} \mu(\gamma) \nu_\gamma$$
where $\nu_\gamma$ is the conditional probability on the set of infinite paths which start with $\gamma$.
By Lemma \ref{lem:getting_iids}, $\varphi_*(\nu_\gamma) = \delta_\gamma \times (\mu_v)^\mathbb{N}$, so 
$$\varphi_*(\nu) = \sum_{v \in \mathcal{R}} \sum_{\gamma \in N_v} \mu(\gamma) \left(\delta_\gamma \times (\mu_v)^\mathbb{N}\right)$$
Now, by definition the pushforward of $\delta_1 \times (\mu_v)^\mathbb{N}$ by the  boundary map bnd is the harmonic measure $\nu_v$, 
hence by $G$-equivariance $\textup{bnd}_\star(\delta_\gamma \times (\mu_v)^\mathbb{N}) = \gamma_\star \nu_v$, 
yielding the claim.
\end{proof}

We can also obtain an analogous statement for any vertex of large growth. Indeed, for each vertex $v$ of large growth, 
there exists a measure $\nu_v$ on $\partial G$ which is the hitting measure 
of the Markov chain on $\partial G$, starting from the identity element on $G$. 

Then we have 
\begin{equation} \label{E:Gcombo}
\nu_v = \sum_{w \in \mathcal{R}} \sum_{\gamma: v \to w} \mu(\gamma) \gamma_* \nu_w
\end{equation}

Here, the sum is over all finite paths from $v$ to $w$ which only meet a recurrent vertex at their terminal endpoint.
Note that if $v$ is recurrent, then $\nu_v$ is the harmonic measure for the random walk on $G$ generated by the measure $\mu_v$, as discussed above.

\subsection{The loop semigroup is nonelementary}

Now suppose that $X$ is a hyperbolic space and that $G \curvearrowright X$ is a nonelementary action. In this section, we show:

\begin{proposition} \label{P:general}
There is a recurrent vertex $v$ of $\Gamma$ such that the corresponding loop semigroup $\Gamma_v$ is nonelementary. 
\end{proposition}

We define the \emph{limit set} of the loop semigroup $\Gamma_v$ on the boundary of $G$ to be $\Lambda_{\Gamma_v} = \overline{\Gamma_v} \cap \partial G$, while the limit set on the boundary of $X$ is $\Lambda^X_{\Gamma_v} = \overline{\Gamma_v x} \cap \partial X$.

For an irreducible component $C$ of $\Gamma$, let $\P_v(C)$ be the set of finite paths in $\Gamma$ which are based at $v$ and lie entirely in $C$, and $P_v(C) = \ev(\P_v(C))$.
Similarly, $\partial \P_v(C)$ is the set of infinite paths in $\Gamma$ which are based at $v$ and lie entirely in $C$, and $\partial P_v(C) = \ev(\partial \P_v(C))$. We remind the reader that $\ev$ is injective on directed paths in $\Gamma$ which start at a fixed vertex $v$.

\begin{lemma} \label{L:open}
There exists a maximal component $C$ of $\Gamma$ and a vertex $v$ in $C$ such that $\partial P_v(C)$ contains an open set on $\partial G$. 
\end{lemma}

\begin{proof}
Every infinite path starting at the origin eventually stays forever in some component, hence we can write $\partial \widetilde{\Gamma}$ 
as the countable disjoint union 
$$\partial \widetilde{\Gamma} = \bigcup_{i = 1}^\infty \gamma_i \cdot \partial \P_{v_i}(C_i)$$
where $\gamma_i$ is a path from the initial vertex to $v_i$, and $C_i$ is the component of $v_i$, hence 
by applying the evaluation map we get
$$\partial G = \bigcup_{i = 1}^\infty g_i \cdot  \partial P_{v_i}(C_i)$$
with $g_i = \ev(\gamma_i)$. 
Since $\partial G$ is a Baire space and all $\partial P_{v_i}(C_i) $ are closed, then there exists some $i$ such that $\partial P_{v_i}(C_i)$ contains 
an open set on $\partial G$. 

For a shadow $S$ in $G$, let $\partial S \subseteq \partial G$ be the set of equivalence classes of geodesic rays based at $1$ which have a representative 
$(g_n)$ which eventually lies in $S$ (i.e. for which there exists $n_0$ such that $g_n \in S$ for all $n \geq n_0$).

If $\partial P_v(C)$ contains a (non-empty) open set, then by definition of the topology on $\partial G$ there exists a nested pair of shadows 
$S = S_1(g, R)$ and $S' = S_1(g, R+r +2\delta)$ 
such that $\partial P_v(C) \supseteq \partial S' \supseteq \partial S  \neq \emptyset$.
Note that by Lemma \ref{l:shadow-large} this implies that $S$ has maximal growth. 
We claim that this implies that for a constant $D$ which depends only on $\delta$ we have
$$S \subseteq N_D(P_v(C)).$$
Then, since $S$ has maximal growth, so does $N_D(P_v(C))$, hence $P_v(C)$ also has maximal growth, proving the lemma. 
To prove the claim, let $h \in S$. Then by Lemma \ref{l:continuable}, there exists $h'$ with $d_G(h,h') \leq r$ and such that $h'$ is continuable.
By Lemma \ref{lem:neigh-shadow}, $h'$ belongs to $S'$. 
Then $h'$ belongs to a geodesic ray $\gamma$ which converges to $\xi \in \partial S$.
Since $\partial S' \subseteq P_v(C)$ and $h'$ belongs to $S'$, then 
there exists $\gamma'$ an infinite geodesic ray which $2\delta$-fellow travels $\gamma$ and which lies in $P_v(C)$. 
Thus, $h$ belongs to a $(2\delta +r)$--neighborhood of $P_v(C)$. Setting $D = 2\delta +r$
establishes the claim.
\end{proof}

\begin{lemma} \label{L:open2}
Let $v$ be a vertex in the irreducible component $C$. Then 
$$\Lambda_{\Gamma_v} \supseteq \partial P_v(C).$$
\end{lemma}

\begin{proof}
Let $\xi \in \partial P_v(C)$, and $(g_n)$ a sequence of elements of $P_v(C)$ so that $g_n \to \xi$. Then there exists for each $n$ a path $s_n$ from 
the endpoint of $g_n$ to $v$, so $g_n s_n \in \Gamma_v$ and $|s_n| \leq D$, where $D$ is the diameter of the component $C$. 
Then $d(g_n s_n, g_n) \leq D$, hence $g_n s_n$ also converges to $\xi$, proving the claim.
\end{proof}
 
 \begin{lemma} \label{L:grows}
If $C$ is a maximal irreducible component and $v$ belongs to $C$, then the growth of $\Gamma_v$ is maximal. 
\end{lemma}

\begin{proof}
Let $D$ be the diameter of $C$.
Since every path from $v$ to any vertex $v_i$ inside $C$ of length $n$ can be extended 
to a path of length $\leq n + D$ from $v$ to itself, we have for each $n$
$$\#(P_v(C) \cap B_n) \leq \#(\Gamma_v \cap B_{n+D})$$
which, together with Lemma \ref{L:growth}, proves the claim.
\end{proof}

We can now give the proof of Proposition \ref{P:general}:

\begin{proof}[Proof of Proposition \ref{P:general}]
Let $v$ be the vertex given by Lemma \ref{L:open}. We first claim that the semigroup $\Gamma_v$ does not have a finite orbit in $\partial X$. Otherwise, the subgroup $H = \langle \Gamma_v \rangle$ also has a finite orbit in $\partial X$. Since $C$ is maximal, by Lemma \ref{L:grows} the semigroup 
$\Gamma_v$ has maximal growth, hence $|B(r) \cap \Gamma_v| \asymp \lambda^r$, where $\lambda$ is the growth rate of $G$; then \cite[Theorem 4.3]{gouezel2015entropy} implies that the subgroup $H$ must have finite index in $G$. This implies that the action $G \curvearrowright \partial X$ has a finite orbit, contradicting that $G \curvearrowright X$ is nonelementary.

Using Proposition \ref{th:Tushar}, it remains to show that $\Lambda^X_{\Gamma_v}$ contains at least $3$ points. To this end, let $U$ be an open set of $\partial G$ contained in $\Lambda_{\Gamma_v} \subset \partial G$, which exists by Lemma \ref{L:open2}. Since the action of $G$ on its boundary is minimal, we can find $f,g \in G$, which are independent loxodromics with respect to the action $G \curvearrowright X$ and with $f^{\pm \infty}, g^{\pm \infty} \in U \subset \partial G$. We claim that each of these $4$ points gives points of $\Lambda^X_{\Gamma_v}$. Since $f,g$ are independent and loxodromic with respect to the action $G \curvearrowright X$, these $4$ points must be distinct points of $\Lambda^X_{\Gamma_v}$, completing the proof.

Since $f^\infty \in U \subset \Lambda_{\Gamma_v}$, there is a sequence $l_i \in \Gamma_v$ with $l_i \to f^\infty$ in $G \cup \partial G$ as $i \to \infty$. As $\widetilde{\Gamma} \cup \partial \widetilde \Gamma$ is compact, after passing to a subsequence, we may assume that $l_i \to l_\infty$ in $\widetilde{\Gamma} \cup \partial \widetilde \Gamma$ and by continuity, $\ev(l_\infty) = f^\infty$. Let $g_i$ be element of $G$ represented by the $i$th term of $l^\infty$ (realizing $l^\infty$ as an infinite path in $\widetilde \Gamma$). Then $(g_i)_{i\ge0}$ is a geodesic in $G$ converging to $f^{\infty} \in \partial G$ and so there is a $K$, depending only on $f$, such that
\[
d_G(g_i, f^{j_i}) \le K,
\]
for some sequence $j_i$ which goes to $\infty$ as $i \to \infty$. Since the (fixed) orbit map $G \to X$ is $L$--coarsely Lipschitz,
\[
d_X(g_i  x, f^{j_i}  x) \le KL. 
\]
As $f \in G$ is loxodromic for the $X$-action,  $f^{j_i}  x$ converges to some point $f^\infty_X \in \partial X$, hence so does $g_i  x$. We conclude $f^\infty_X \in \Lambda^X_{\Gamma_v}$, as required.
\end{proof}

\begin{remark}
If it is a priori known that no element $g \in G$ acts parabolically on $X$ (i.e. $g$ is either loxodromic or has a bounded orbit), then the use of \cite[Theorem 4.3]{gouezel2015entropy} in Proposition \ref{P:general} (and in Corollary \ref{c:general}) can be removed.
\end{remark}

\subsection{Convergence to the boundary of $\partial X$}

In this section we show that almost every sample path for the Markov chain converges to the boundary of $X$. Since sample paths in the Markov chain evaluate to geodesic rays in $G$, this will show that the orbit of almost every geodesic ray in $G$ converges to the boundary of $X$.

Let $\partial \widetilde \Gamma^X$ be the set of infinite paths beginning at $v_0$ which converge to a point in $\partial X$ when projected to $X$ and set $\partial ^X G = \ev(\partial \widetilde \Gamma^X)$. Since the orbit map $G \to X$ is Lipschitz, $\partial^X G \subset \partial G$ consists of those $\xi$ such that the projection of any geodesic ray $[1,\xi) \subset G$ to $X$ converges in $X \cup \partial X$ to a point of $\partial X$.

\begin{lemma}
The set $\partial^X G$ is $G$--invariant.
\end{lemma}

\begin{proof}
If $\eta \in \partial^X G$, then $[1,\eta)$ is a geodesic in $G$ which projects under the orbit map to a path $\gamma$ in $X$ which converges to a point $\xi$ in $\partial X$. Hence, $g \gamma$ converges to $g \xi$ in $\partial X$. Since $g[1,\eta)$ is a geodesic in $G$ with the same endpoint as $[1,g\eta)$, they $R$--fellow-travel in $G$. Hence, the projection $\gamma'$ of $[1,g\eta)$ to $X$ must $RL$--fellow-travel the path $g\gamma$ in $X$, where $L$ is the Lipschitz constant of the fixed orbit map $G \to X$. Hence, $\gamma'$ converges to the point $g\xi$ in $X \cup \partial X$ showing that $g\eta \in \partial^X G$ as required.
\end{proof}

Here we show that $\partial^X G$ has full measure in $\partial G$, so that the map
$\Phi \colon \partial^X G \to \partial X$ gives a measurable map $\partial G \to \partial X$.

\begin{theorem} \label{th:chain_converges}
For $\mathbb{P}$-almost every path $(w_n)$ in the Markov chain, the projection $(w_n x)$ to the space converges to a point in the boundary $\partial X$. 
\end{theorem}

\begin{proof}
We know by Lemma \ref{L:recurrent} that $\mathbb{P}$-almost every sample path enters some recurrent component $C$, and stays there forever. Let $v$ be a vertex in $C$.
Then for each $v$, one can split the infinite path into a prefix $\gamma_0$, and a sequence $\gamma_1, \gamma_2, \dots$ of primitive loops at $v$. 
Then the pushforward of $\mathbb{P}_C$ via the map $\varphi_v$ equals the measure $\mu_v^{\mathbb{N}}$ on the set $\Gamma_v^{\mathbb{N}}$ (as in Theorem \ref{prop:PS_comb}). 
Now, if the measure $\mu_v$ is non-elementary, then by Theorem \ref{thm:MT_drift} we know that almost surely the sequence 
$(\gamma_1 \gamma_2 \dots \gamma_n x)$ converges to a point in $\partial X$. 
Since the same is true for every vertex of $C$, then the sequence $(w_n)$ is partitioned in a finite number of subsequences, and for each such subsequence $(w_{n_k})$ 
the sample path $(w_{n_k}x)$ converges to a point in $\partial X$. 
Thus, we define the equivalence relation on the set of vertices of $C$ by defining $v_i \sim v_j$ if the sequences $(w_{n(k,v_i)}x)$ and $(w_{n(k, v_j)}x)$ 
converge to the same point. 
Now, we know that if $w_{n_k}x \to \xi \in \partial X$, then also $w_{n_k+1}x \to \xi$, so this equivalence relation satisfies the hypothesis of Lemma \ref{L:equiv}, 
hence as a consequence there is only one equivalence class, thus the whole sequence $(w_n x)$ converges to the same $\xi \in \partial X$.

Since there exists at least one (recurrent) vertex $v$ such that the semigroup $\Gamma_v$ is nonelementary by Proposition \ref{P:general},
the previous argument shows that $\nu(\widetilde \Gamma^X) >0$, thus $\nu(\partial^X G)> 0$. Since $G \curvearrowright (\partial G, \nu)$ is ergodic (Lemma \ref{L:ergodic}) and $\partial^X G$ is invariant, then $\partial^X G$ has full measure, which completes the proof of the theorem.
\end{proof}

\begin{lemma} \label{L:equiv}
Let $r\geq 1$, and $\sim$ be an equivalence relation on the set $\{1, \dots, r\}$, and $\mathbb{N} = A_1 \sqcup \dots \sqcup A_r$ a partition 
of the set of natural numbers into finitely many, disjoint infinite sets. Assume moreover that if 
the intersection $(A_i + 1) \cap A_j$ is infinite, then $i \sim j$. Then there is only one equivalence class.
\end{lemma}

\begin{proof}
Suppose that there are at least two equivalence classes, and pick one equivalence class. Let $X$ be the union of all sets $A_i$ which belong to such a class, 
and $Y$ the union of all other equivalence classes. Then $\mathbb{N} = X \sqcup Y$, with both $X$ and $Y$ infinite. Then $(X+1) \cap Y$ is infinite, 
which implies that there exists a set $A_i \subseteq X$ and a set $A_j \subseteq Y$ such that $|(A_i + 1) \cap A_j| = \infty$, contradicting the fact 
that $i$ and $j$ are in different equivalence classes.
\end{proof}

Theorem \ref{th:intro_1} now follows immediately as a corollary to Theorem \ref{th:chain_converges}.
\begin{corollary}[Convergence to the boundary of $X$]
For every $x \in X$ and $\nu$--almost every $\eta \in \partial G$, if $(g_n)_{n\ge0}$ is a geodesic in $G$ converging to $\eta$, then the sequence $g_nx$ in $X$ converges to a point in the boundary $\partial X$.
\end{corollary}

Theorem \ref{th:chain_converges} was proven using only the knowledge that for \emph{some} recurrent vertex $v$, $\Gamma_v$ was nonelementary. However, using Theorem \ref{th:chain_converges} we can now show that for \emph{every} recurrent vertex, $\Gamma_v$ is nonelementary. This will be crucial for what follows.

\begin{corollary} \label{c:general}
Suppose that $G \curvearrowright X$ is as above. Then for each recurrent vertex $v$ of 
$\Gamma$, 
the semigroup $\Gamma_v$ is nonelementary with respect to the action on $X$.
\end{corollary}

\begin{proof}
Since $\Gamma_v$ has maximal growth, $\nu(\partial \P_v(C))>0$. Then, identifying $\partial \P_v(C)$ with the corresponding endpoints in $\partial \widetilde \Gamma$, we have that $\partial \P_v(C) \cap \partial \widetilde \Gamma^X$ is infinite since $\nu(\partial \widetilde \Gamma^X)=1$. Hence, so is $\ev(\partial \P_v(C) \cap \partial \widetilde \Gamma^X) \subset \Lambda_{\Gamma_v} \cap \partial ^X G$ and we conclude that $\Lambda^X_{\Gamma_v}$ is infinite. 

Just as in the proof of Proposition \ref{P:general}, the subgroup $H = \langle \Gamma_v \rangle$ has finite index in $G$ and so $\Gamma_v$ cannot have a finite orbit on $\partial X$. Since $\Lambda^X_{\Gamma_v} \neq \emptyset$, Theorem \ref{th:Tushar}, implies that $\Gamma_v$ is nonelementary.
\end{proof}

Similar to the discussion in Section \ref{sec: PS_comb}, we denote as $\nu_v^X$ the corresponding hitting measure on $\partial X$ (which is well-defined since we already 
proved the Markov process converges to the boundary).
By equation \eqref{E:Gcombo} and applying convergence to the boundary we get: 

\begin{equation} \label{E:combo}
\nu^X_v = \sum_{w \in \mathcal{R}} \sum_{\gamma: v \to w} \mu(\gamma) \gamma_* \nu^X_w
\end{equation}

\begin{lemma} \label{lem:nonatomic}
For any $v$ of large growth, the measure $\nu_v^X$ is non-atomic.
\end{lemma}

\begin{proof}
Since the random walk measures $\nu^X_w$ are non-atomic, so are the measures $\gamma_* \nu^X_w$ for each $\gamma$, 
hence by equation \eqref{E:combo} the measure $\nu_v$ is also non-atomic as it is a linear combination of non-atomic measures. 
\end{proof}

\subsection{Positive drift along geodesics}
In this section, we show that almost every sample path has positive drift in $X$. This will imply (Corollary \ref{cor:PS_drift}) the same for almost every geodesic ray in $G$ with respect to Patterson--Sullivan measure. 

We first show that there is a well-defined average return time for a recurrent vertex $v$ of $\Gamma$:

\begin{lemma} \label{l:return}
For each vertex $v$ in $C$ recurrent, there exists $T_v > 0$ such that for almost every $\omega \in \Omega_C$
$$\lim_{k \to \infty} \frac{n(k,v,\omega)}{k} = T_v.$$
\end{lemma}

\begin{proof}
We know by Lemma \ref{lem:first_return} 
that the return time to $v$ has an exponential tail:
$$\mathbb{P}_v(\tau(1, v, \omega) \geq n) \leq c e^{-cn}$$
Thus, the average return time is finite: 
$$\int \tau(1,v, \omega) \ d\mathbb{P}_v(\omega)= T_v < \infty$$
which implies that $\mu_v$ has finite first moment in the word metric: 
$$\int |g| \ d \mu_v(g) = T_v.$$
Now, by definition 
$$n(k, v, \omega) = \tau(0, v, \omega) + \tau(1, v, \omega) + \dots + \tau(k, v, \omega)$$
then since the variables $\tau(1, v, \omega), \dots, \tau(k, v, \omega)$ are independent and identically distributed, 
we get almost surely 
$$\lim_{k \to \infty} \frac{n(k,v,\omega)}{k} = T_v.$$
\end{proof}

Using these average return times, we show that sample paths in the Markov chain have a well-defined drift and hence make linear progress in $X$.

\begin{theorem} \label{th:drift}
There exists $L > 0$ such that for $\mathbb{P}$-almost every sample path $(w_n)$ we have 
$$\lim_{n \to \infty} \frac{d(w_n x, x)}{n} = L.$$
\end{theorem}

\begin{proof}
Let $v$ be a recurrent vertex. Then by Proposition \ref{c:general} the loop semigroup $\Gamma_v$ is nonelementary, hence the random walk 
given by the return times to $v$ has positive drift. More precisely, 
from Theorem \ref{thm:MT_drift}, there exists a constant $\ell_v > 0$ such that 
for almost every sample path which enters $v$,
$$\lim_{k \to \infty} \frac{d(w_{n(k,v)} x, x)}{k} = \ell_v.$$
Combined with Lemma \ref{l:return}, this implies 
$$\lim_{k \to \infty} \frac{d(w_{n(k,v)} x, x)}{n(k,v)} = \frac{\ell_v}{T_v}.$$
Now, almost every infinite path visits every vertex of some recurrent component infinitely often.
Thus, for each recurrent vertex $v_i$ which belongs to a component $C$, there exists a constant $L_i > 0$ such that for $\mathbb{P}_C$-almost every path $(w_n)$, there is a limit 
$$L_i = \lim_{k \to \infty} \frac{d(w_{n(k, v_i)} x, x)}{n(k, v_i)}.$$
Let $C$ be a maximal component, and $v_1, \dots, v_k$ its vertices. Our goal now is to prove that $L_1 = L_2 = \dots L_k$. 
Let us pick a path $\omega \in \Omega_0$ such that the limit $L_i$ above exists for each $i = 1, \dots, k$, and define $A_i = \{ n(k, v_i), k \in \mathbb{N} \}$, and the equivalence relation $i \sim j$ if $L_{i} = L_{j}$. 
Since $w_{n(k, v_i)}$ and $w_{n(k, v_i) + 1}$ differ by one generator, $d(w_{n(k, v_i)}x, w_{n(k, v_i) + 1}x)$ is uniformly bounded, hence
$$\lim_{k \to \infty} \frac{d(w_{n(k, v_i) + 1} x, x)}{n(k, v_i) + 1} = \lim_{k \to \infty} \frac{d(w_{n(k, v_i)}x, x)}{n(k, v_i)} = L_i$$
so the equivalence relation satisfies the hypothesis of Lemma \ref{L:equiv}, 
hence there is a unique limit $L = L_i$ so that 
$$\lim_{n \to \infty} \frac{d(w_n x, x)}{n} = L.$$


In order to prove that the drift is the same for all maximal components, 
let us now pick $L = L_i$ for some recurrent vertex $v_i$, and consider the set $\partial G_L$ of points $\xi$ in $\partial G$ such that 
there exists a geodesic $(g_n) \subseteq G$ with $g_n \to \xi$ and such that 
$\lim_{n \to \infty} \frac{d(g_n x, x)}{n} = L$.
The set $\partial G_L$ is $G$-invariant, and by the above statement it has positive probability, hence by ergodicity it has full measure. 
\end{proof}

\begin{theorem} \label{th:drift2}
For every vertex $v$ of large growth, and for $\mathbb{P}_v$-almost every sample path 
$(w_n)$ we have 
$$\lim_{n \to \infty} \frac{d(w_n x, x)}{n} = L.$$
\end{theorem}

\begin{proof}
By Theorem \ref{th:drift}, for $\mathbb{P}$-almost every path which passes through $v$, 
the drift equals $L$. Each such path $\gamma = (w_n)$ can be decomposed as $\gamma = \gamma_0 \star \widetilde{\gamma}$, 
where $\gamma_0$ is a finite path from $v_0$ to $v$ and $\widetilde{\gamma} = (\widetilde{w}_n)$ is an infinite path starting from $v$. 
 Thus for each $n$, $w_n = g_0 \widetilde{w}_n$, where $g_0$ is the group 
element which represents $\gamma_0$.
Hence, 
\[
d(w_nx,x) - d(x,g_0x) \le d(\widetilde{w}_nx , x) = d(w_nx, g_0 x) \le d(w_nx,x) + d(x,g_0x)
\]
and so 
$$\lim_{n \to \infty} \frac{d(\widetilde{w}_n x, x)}{n} = \lim_{n \to \infty} \frac{d(w_n x, x)}{n} = L$$
as required.
\end{proof}

For application in Section \ref{sec:counting}, we will need the following convergence in measure statement.
\begin{corollary} \label{cor:drift_in_measure}
For any $\epsilon > 0$, and for any $v$ of large growth, 
$$\mathbb{P}_v\left(\left| \frac{d(x, w_n x)}{n} - L \right| \geq \epsilon \right) \to 0.$$
\end{corollary}

\begin{proof}
By the theorem, the sequence of random variables $X_n = \frac{d(w_nx, x)}{n}$ converges almost surely to $L$. 
Moreover, for every $n$ the variable $X_n$ is bounded above by the Lipschitz constant of the orbit map $G \to X$.
Thus, $X_n$ converges to $L$ in $L^1$, yielding the claim. 
\end{proof}

Finally, Theorem \ref{th:drift} implies the following linear progress statement for almost every geodesic ray in $G$. This proves Theorem \ref{th:intro_2} from the introduction.

\begin{corollary}[Positive drift] \label{cor:PS_drift}
For every $x \in X$ and $\nu$--almost every $\eta \in \partial G$, if $(g_n)_{n\ge0}$ is a geodesic in $G$ converging to $\eta$, then 
\[
\lim_{n \to \infty} \frac{d_X(x, g_n x)}{n} = L >0.
\]
\end{corollary}

\begin{proof}
Note that the measure $\mathbb{P}$ on the set of infinite paths starting at $1$ pushes forward to $\nu$ on $\partial G$. 
Thus, by Theorem \ref{th:drift}, for $\nu$-almost every $\xi \in \partial G$ there exists a geodesic $(g_n)$ in the combing with $g_n \to \xi$ 
for which the statement is true; by hyperbolicity, any other geodesic $(g_n')$ which converges to $\xi$ fellow travels $(g_n)$, hence the drift in $X$ is the same.
\end{proof}

\subsection{Decay of shadows for $\mathbb{P}$}

Let us denote as $Sh(x,r)$ the set of shadows of the form $S_x(gx, R)$ where $g \in G$ and the distance parameter satisfies $d(x,gx) - R \geq r$. For any shadow $S$, we denote its closure in $X \cup \partial X$ by $\overline S$.

\begin{proposition}[Decay of shadows] \label{P:decay1}
There exists a function $f : \mathbb{R} \to \mathbb{R}$ such that $f(r) \to 0$ as $r \to \infty$, and such that 
for each $v$ of large growth, 
$$\nu_v^X \big(\overline{S_x(gx, R)}\big) \leq f(r),$$
where $r = d(x, gx) -R$ is the distance parameter of the shadow. 
\end{proposition}

\begin{proof}
By Lemma \ref{lem:nonatomic}, the measure $\nu_v^X$ on $\partial X$ is nonatomic. The proposition now follows from a standard measure theory argument, as in Proposition 5.1 of \cite{MaherTiozzo}.
\end{proof}

Let $\Omega$ be the set of all infinite paths in $\Gamma$, starting from any vertex. 
We define for each $n$ the map $w_n : \Omega \to G$ which maps each infinite path to the 
product of the first $n$ edges of the path: 
$$w_n((e_1, \dots, e_n, \dots)) = e_1 \dots e_n$$
and 
$$w_\infty(x) := \lim_{n \to \infty} w_n x \in \partial X$$
(if it exists). Recall that by definition $\nu_v^X(A) = \mathbb{P}_v(w_\infty(x) \in A)$.

\begin{proposition}  \label{p:decay}
There exists a function $p : \mathbb{R}^+ \to \mathbb{R}^+$ with $p(r) \to 0$ as $r \to \infty$, such that for each vertex $v$
and any shadow $S_x(gx, R)$ we have 
\[
\mathbb{P}_v \Big( \exists n  \geq  0 \ : \ w_n x \in S_x(gx, R) \Big)  \leq p(r),
\]
where $r = d(x, gx) - R$ is the distance parameter of the shadow.
\end{proposition}

\begin{proof}
Let $S = S_x(gx, R)$ and $S_1 = S_x(gx, R + c_1)$, for $c_1$ to be determined in a moment. The idea is that for $c_1$ sufficiently large, the probability that the sample path $(w_nx)$ ever reaches the shadow $S$ is dominated by the probability that it convergence to a point of $\overline S_1 \cap \partial X$.

For any $h\in G$ such that $hx \in S$, hyperbolicity of $X$ implies that there is a constant $c_2 \ge0$ with $c_2 = c_1 + O(\delta)$ such that the complement of $S_1$ is contained in the shadow
\[
S^h = S_{hx} \big( x, d(x,hx) - c_2 \big).
\]
Note that the distance parameter of $S^h$ is $c_2 = c_1 + O(\delta)$. Then since $h^{-1}S^h$ is a shadow based at $x$, Proposition \ref{P:decay1} implies that we can choose $c_1$ large enough so that 
\[
\nu_w^X(h^{-1}S^h) \le f(c_1 + O(\delta)) \le 1/2,
\]
for any large growth vertex $w$. This is our fixed $c_1 \ge 0$.

Now we have 
$$\mathbb{P}_v( w_\infty(x) \in \overline{S_1}) \geq \mathbb{P}_v( w_\infty(x) \in \overline{S_1} \ \textup{ and } \ \exists n  \ : \ w_n x \in S ) = $$
If there exists $n$ such that $w_n x$ belongs to $S$, let $n_1$ be the first hitting time. Then 
$$ = \sum_{hx \in S} \mathbb{P}_v(w_{n_1} = h) \mathbb{P}_v(w_\infty(x) \in \overline{S_1} \  | \ w_{n_1} = h)  = $$
hence from the Markov property of the chain
$$ = \sum_{hx \in S} \mathbb{P}_v(w_{n_1} = h) \mathbb{P}_{v^h}(w_\infty(x) \in h^{-1} \overline{S_1}),$$ 
where $v^h$ is terminal vertex of the path starting at $v$ which reads the words $h$.


Since the complement of $S_1$ is contained in $S^h$, $(\partial X - h^{-1}\overline S_1) \subset h^{-1} \overline{S^h}$ and so
$$\nu^X_w(h^{-1} \overline{S_1}) \geq 1 - \nu^X_w(h^{-1} \overline{S^h}) \geq 1 - f(c_1 + O(\delta)) \geq 1/2 ,$$
by our fixed choice of $c_1$.
Hence, combining this with the previous estimate, 
$$\mathbb{P}_v( w_\infty(x) \in \overline{S_1}) \geq \sum_{hx \in S} \mathbb{P}_v(w_{n_1} = h) \cdot \nu^X_{v^h}(h^{-1} \overline{S_1}) \geq \frac{1}{2}
 \mathbb{P}_v(\exists n \ : \ w_n x \in S)$$ 
hence 
$$\mathbb{P}_v \big( \exists n \ : \ w_n x \in S \big) \leq 2 \cdot \mathbb{P}_v( w_\infty(x) \in \overline{S_1})  \leq 2 \cdot f\big(d(x, gx) - R -c_1\big)$$ 
so the claim is proven by taking $p(r) := 2 \cdot f(r - c_1)$.
\end{proof}

We conclude this section with a simple lower bound on the Gromov product:
\begin{lemma}\label{l:gromov}
Let $x,y,z$ be points in a metric space $X$. Then 
$$(y,z)_x \geq d(x, z) - d(y,z).$$
\end{lemma}

\begin{proof}
By definition 
$$(y,z)_x = \frac{1}{2}\left( d(x,y) + d(x, z) - d(y,z) \right)$$
and by the triangle inequality $d(x, y) \geq d(x, z) - d(y,z)$, yielding the claim.
\end{proof}

\subsection{Almost independence and Gromov products}
Let $\mathbb{P}$ denote the probability measure (on the space of infinite paths) for the Markov chain starting at the identity. 
Given a vertex $v$, we will denote as $\mathbb{P}_v(w_n = g)$ the probability that the Markov chain starting at the vertex $v$ 
reads $g$ after $n$ steps.
Let $m = \lfloor n/2 \rfloor$.
Finally, we denote as $u_m$ the random variable $u_m = w_m^{-1} w_{n}$, so that $w_{n} = w_m u_m$.
Recall that by the Markov property of the chain, for any $g,h \in G$ and each $m$, 
$$\mathbb{P}(w_m = g \textup{ and } u_m^{-1} = h) = 
 \mathbb{P}(w_m = g) \cdot \mathbb{P}_v(w_m = h^{-1})$$
 where $v = [g]$.
Let $V$ denote the vertex set of $\Gamma$.


The following three lemmas will be essential in controlling the size of $(gx,g^{-1}x)_x$ over the Markov chain (Lemma \ref{lem:prod_inv}), which is needed to produce loxodromic elements of the action $G \curvearrowright X$.

\begin{lemma} \label{lem:prob_prod_1}
Let $f : \mathbb{R} \to \mathbb{R}$ be any function such that $f(m) \to +\infty$ as $m \to +\infty$. Then 
$$\mathbb{P} \big((w_m x, u_m^{-1} x)_x \geq f(m) \big) \to 0$$
as $m \to \infty$.
\end{lemma}

\begin{proof} Let us denote as $S(m) := \{ (g,h) \in G \times G \ :  (gx, hx)_x \geq f(m) \}$. Then by definition we have 
$$\mathbb{P}((w_m x, u_m^{-1} x)_x \geq f(m) ) = \sum_{g,h \in S(m)} \mathbb{P}(w_m = g \textup{ and } u_m = h^{-1} ) \leq $$
hence by the Markov property 
$$ \leq \sum_{v \in V} \sum_{h \in G} \mathbb{P}_v(w_m = h^{-1} ) \sum_{\{g : [g] =v, (g,h) \in S(m)\}} \mathbb{P}(w_m = g) \leq $$
and forgetting the requirement that $[g] = v$
$$\leq \sum_{v \in V} \sum_{h \in G} \mathbb{P}_v(w_m = h^{-1} ) \mathbb{P}( w_m x \in S_x(hx, R)) \leq $$
where $R = d(x, h x) - f(m)$, while using the estimate on shadows (Proposition \ref{p:decay}) we have 
$$\leq \sum_{v \in V} \sum_{h \in G} \mathbb{P}_v(w_m = h^{-1} ) p(f(m)) \leq \#V \cdot p(f(m))$$
and the claim follows by decay of shadows. 
\end{proof}

\begin{lemma} \label{lem:prob_prod_2}
For each $\eta > 0$, the probability
$$\mathbb{P} \big((w_n^{-1} x, u_m^{-1} x)_x \leq n(L - \eta)/2 \big)$$
tends to $0$ as $n \to \infty$.
\end{lemma}

\begin{proof}
By Lemma \ref{l:gromov}, and since the group acts by isometries,
$$(w_n^{-1} x, u_m^{-1} x)_x \geq d(x, w_n^{-1}x) - d(w_n^{-1} x, u_m^{-1} x) =  d(w_n x, x) - d(x, w_m x)$$
Now, by Theorem \ref{th:drift}, 
$$\mathbb{P}(d(w_n x, x) \geq n(L-\epsilon)) \to 1 \quad \text{and} \quad \mathbb{P}(d(w_m x, x) \leq n/2(L+\epsilon)) \to 1$$
hence with probability which tends to $1$ 
$$(w_n^{-1} x, u_m^{-1} x)_x \geq n(L-\epsilon) - \frac{n}{2}(L+\epsilon) = \frac{n}{2}(L-3\epsilon)$$
which yields the claim setting $\eta = 3\epsilon/2$.

\end{proof}

Finally, we show 
\begin{lemma} \label{lem:prob_prod_3}
For each $\eta > 0$, the probability
$$\mathbb{P} \big( (w_m x, w_n x)_x \leq n (L - \eta)/2 \big)$$
tends to $0$ as $n \to \infty$.
\end{lemma}

\begin{proof} By Lemma \ref{l:gromov}, 
$$(w_n x, w_m x)_x \geq d(x, w_n x) - d(w_m x, w_n x) = d(x, w_n x) - d(x, u_m x) $$
so if we have $d(x, w_n x) \geq n(L-\epsilon)$ and $d(x, u_m x) \leq n(L+\epsilon)/2$ then 
$$(w_n x, w_m x)_x \geq  n(L-3\epsilon)/2 $$
Now, the first statement holds with probability which tends to $1$ by Theorem \ref{th:drift}; 
for the second statement, note that for each $g$ and each $m$, 
$$\mathbb{P}(u_m = g) =  \sum_{v \in V} \mathbb{P}([w_m] = v) \mathbb{P}_v(w_m = g)$$
and by Theorem \ref{th:drift2} 
$$\mathbb{P}_v \left(d(x, w_m x) \leq n(L+\epsilon)/2 \right) \to 1$$
 for each recurrent $v$.
Hence $$\mathbb{P} \left(d(x, u_m x) \leq n(L+\epsilon)/2 \right) \to 1$$
completing the proof.
\end{proof}



\subsection{Linear translation length along geodesics}
For $g \in G$, let $\tau_X(g)$ be the (stable) translation length of $g$ with respect to its action on $X$. Recall that $\tau_X(g) > 0$ if and only if $g$ is a loxodromic for the action $G \curvearrowright X$.

The following lemma for estimating the translation length of an isometry $g$ of $X$ is well-known; see for example \cite[Proposition 5.8]{MaherTiozzo}.

\begin{lemma} \label{l:tau-formula}
There exists a constant $c$, which depends only on $\delta$, such that for any isometry $g$ of a $\delta$-hyperbolic space $X$ 
and any $x \in X$ with $d(x, gx) \geq 2(gx, g^{-1}x)_x + c$, the translation length of $g$ is given by 
$$\tau_X(g) = d(x, gx) - 2(gx, g^{-1}x)_x + O(\delta).$$
\end{lemma}

Since we already proved that the term $d(x, gx)$ grows linearly $\mathbb{P}$--almost surely (Theorem \ref{cor:drift_in_measure}), to complete the proof we need to 
show that the Gromov product $(gx, g^{-1}x)_x$ does not grow too fast. To do so, we will use the following trick. For a proof, see \cite{TT}.

\begin{lemma}[Fellow traveling is contagious]\label{l:fellow_travel}
Suppose that $X$ is a $\delta$--hyperbolic space with basepoint $x$ and suppose that $A\ge0$. If $a,b,c,d \in X$ are points of $X$ with $(a\cdot b)_{x} \ge A$, $(c\cdot d)_{x} \ge A$, and $(a\cdot c)_{x} \le A-3\delta$. Then $(b \cdot d)_{x} -2\delta \le (a \cdot c)_{x} \le (b \cdot d)_{x} +2\delta$.
\end{lemma}

We apply this to prove the following:

\begin{lemma} \label{lem:prod_inv}
Let $f : \mathbb{N} \to \mathbb{R}$ be any function such that $f(n) \to +\infty$ as $n \to +\infty$. Then 
$$\mathbb{P} \Big( (w_n x, w_n^{-1} x)_x \geq f(n) \Big) \to 0$$
as $n \to \infty$.
\end{lemma}

\begin{proof}
Define 
$$f_1(n) = \min \left\{ f(n) - 2 \delta, \frac{n(L-\eta)}{2} - 3\delta \right\}$$
It is easy to see that $f_1(n) \to \infty$ as $n \to \infty$.
By Lemma \ref{l:fellow_travel}, if we know that:
\begin{enumerate}
\item
 $(w_mx, w_nx)_x \geq n(L-\eta)/2$, 
 \item
 $(u_m^{-1}x, w_n^{-1}x)_x \geq n(L-\eta)/2$, and
\item
 $(w_mx, u_m^{-1}x)_x \leq f_1(n) \leq n(L-\eta)/2 - 3\delta$, 
\end{enumerate}
then 
$$(w_nx, w_n^{-1}x)_x \leq (w_mx, u_m^{-1}x)_x + 2 \delta \leq f_1(n) + 2\delta.$$
Using Lemmas \ref{lem:prob_prod_1}, \ref{lem:prob_prod_2}, and \ref{lem:prob_prod_3}, the probability that conditions (1),(2), (3) hold tends to $1$, hence we have 
$$\PP^n \left( (w_nx, w_n^{-1}x)_x \leq f(n) \right) \to 1$$
as $n \to \infty$.
\end{proof}

We can now prove Theorem \ref{th:intro_Mark} from the introduction.
\begin{theorem}
For every $\epsilon > 0$, one has 
\[
\PP \Big(\tau_X(w_n) \ge n (L-\epsilon) \Big) \to 1,
\]
as $n \to \infty$.
\end{theorem}

\begin{proof}
If we set $f(n) = \eta n$ with $\eta > 0$, then by Lemma \ref{lem:prod_inv} and Corollary \ref{cor:drift_in_measure} the events 
$(w_nx, w_n^{-1}x)_x \leq \eta n$ and $d(x, w_nx) \geq n(L-\eta)$ occur with probability ($\mathbb{P}^n$) which tends to $1$, hence 
by Lemma \ref{l:tau-formula} 
\begin{align*}
\PP^n \Big( \tau(w_n) \ge n(L- 3 \eta)  \Big) \ge 
\PP^n \Big( d(x, w_nx) - 2(w_nx, w_n^{-1}x)_x + O(\delta) \ge n(L- 3 \eta) \Big) 
\end{align*}
which approaches $1$ as $n \to \infty$. This implies the statement if we choose $\epsilon > 3 \eta$.
\end{proof}



\section{Generic elements and the action $G \curvearrowright X$}
\label{sec:counting}

Recall that for any $A \subseteq G$,
\[
P^n(A) = \frac{\#(A \cap S_n)}{\#S_n},
\]
where $S_n$ is the $n$--sphere in $G$. Hence, $P^n$ is uniform measure on the $n$--sphere of $G$.

We remind the reader that there is a constant $c>1$ such that $c^{-1} \lambda^n \le \#S_n \le c \lambda^n$ (Theorem \ref{th:Coorn}) where $\lambda$ is the growth of the directed graph $\Gamma$. Also, since $\ev \colon \widetilde \Gamma \to G$ is bijective, we may unambiguously use graph terminology when referring to elements of $G$. For example, $[g]=v$ means that $v$ is the vertex of $\Gamma$ which is the terminal endpoint of the path starting at $v_0$ and spelling $g$. Further, for $g \in G$, $\widehat{g}$ denote the element of $G$ along the unique combing geodesic from $1$ to $g$ which has distance $\log|g|$ from $g$.

\begin{lemma} \label{l:shadow-cone}
For any $R \geq 0$ and any $h \in G$, the group shadow $S_1(h, R) \subseteq G$ is contained in the union of finitely many cones: 
\[
S_1(h,R) \subseteq \bigcup_{g \in B_{R+c}(h)} \cone(g)
\]
where $c$ depends only on $\delta$.
\end{lemma}

\begin{proof}
For each $x, y, z$ in a metric space, by definition of shadow and Gromov product,
$$z \in S_x(y, R) \Leftrightarrow (y,z)_x \geq d(x, y) - R \Leftrightarrow (x,z)_y \leq R$$
Thus, if $g_1 \in S_1(h, R)$, then $(g_1, 1)_h \leq R$, which implies that the distance between any geodesic segment from $1$ to $g_1$ 
and $h$ is at most $R + c$, where $c$ only depends on $\delta$. In particular, there exists a group element $h_1$ on the 
combing geodesic from $1$ to $g_1$ which lies at distance $\leq R + c$ from $h$. Thus, 
$g_1 \in \cone(h_1)$ with $h_1 \in B_{R+c}(h)$.
\end{proof}

\begin{lemma}\label{lem:near_large_growth}
There is $K\ge0$ such that for every $h \in G$ there is a $g \in G$ such that $v=[g]$ has large growth and $d_G(h,g) \le K$.
\end{lemma}

\begin{proof}
Recall that there exists an $R\ge0$ such that for any $h$ the group shadow $\overline {S_1(h,R)}$ contains an open set on $\partial G$. Hence, $\#(S_1(h,R) \cap S_n) \asymp c_1 \lambda^n$ for some $c_1$ depending only on $h$ (Lemma \ref{l:shadow-large}). Moreover, by Lemma \ref{l:shadow-cone} the shadow $S_1(h, R)$
is contained in a finite union of cones:
\[
S_1(h,R) \subseteq \bigcup_{g \in B_{R+c}(h)} \cone(g)
\]
where $C$ depends only on $\delta$. We conclude that one such cone $\cone(g)$ also has large growth. This completes the proof.
\end{proof}


\subsection{Genericity of positive drift}
The following result establishes Theorem \ref{th: intro_gen_drift} from the introduction using what we have shown about the Markov measures in Section \ref{sec:Markov_action}.

\begin{theorem} \label{th:gen_drift}
For every $\epsilon > 0$ one has 
$$\frac{\#\{ g \in S_n \ : \ d(gx,x) \geq (L- \epsilon) \ |g| \}}{\#S_n} \to 1 \qquad \textup{as }n \to \infty.$$
\end{theorem} 

\begin{proof}
Let $A_L$ denote the set of group elements 
$$A_L :=  \{ g \in G \ : \ d(gx,x) \leq L |g| \}. $$
We know by Corollary \ref{cor:drift_in_measure} that for any $L' < L$ one has 
$$\mathbb{P}^n(A_{L'}) \to 0 \qquad \textup{as }n \to \infty.$$
Then we have 
$$P^n(A_{L-\epsilon}) \leq \frac{\# \{ g \in S_n \ : \ \widehat{g} \notin LG \}}{\#S_n} + \frac{\#\{ g \in S_n \cap A_{L-\epsilon} \ :  \ \widehat{g} \in LG \}}{\#S_n}$$
and we know by Proposition \ref{P:smallg} that the first term tends to $0$. 
Now, by writing $g = \widehat{g}h$ with $| h | = \log |g|$ we have that $d(gx, x) \leq (L-\epsilon) |g|$ implies 
$$d(\widehat{g}x, x) \leq d(gx,x) + d(\widehat{g}x, gx) \leq (L-\epsilon) |g| + d(x, hx) \leq $$
hence, there exists $C$ such that it is less than 
$$ \leq (L-\epsilon) |g| + C \log |g| \leq L'  |\widehat{g}|$$
for any $L-\epsilon < L' < L$ whenever $|g|$ is sufficiently large. This proves the inclusion
$$\{ g \in S_n \cap A_{L-\epsilon} \ :  \ \widehat{g} \in LG \} \subseteq \{ g \in S_n \ :  \ \widehat{g} \in A_{L'} \cap LG \}$$
and by Lemma \ref{L:growth} (1) 
$$\#\{ g \in S_n \ :  \ \widehat{g} \in A_{L'} \cap LG \} \leq c \lambda^{\log n} \#( S_{n - \log n} \cap A_{L'} \cap LG) \leq $$
hence by Lemma \ref{lem:comparePn} (equation \eqref{E:countmu}) and considering the size of $S_{n-\log n}$
$$\leq c_1 \lambda^{\log n} \mathbb{P}^{n-\log n}(A_{L'}) \# S_{n-\log n} \leq c_2 \lambda^n \mathbb{P}^{n-\log n}(A_{L'}).$$
Finally, using that $\mathbb{P}^{n-\log n}(A_{L'}) \to 0$ we get
$$\limsup_{n \to \infty} \frac{\#\{ g \in S_n \cap A_{L-\epsilon} \ :  \ \widehat{g} \in LG \}}{\#S_n} \leq \limsup_{n \to \infty} c_3 \mathbb{P}^{n-\log n}(A_{L'})  = 0$$
which proves the claim.
\end{proof}


\subsection{Counting and decay of shadows}
For $g \in G$, we set
 \[
 S_x^G(gx,R) = \left \{ h\in G : hx \in  S_x (gx,R) \right \}, 
 \]
 where as usual, $S_x (gx,R)$ is the shadow in $X$ around $gx$ centered at the basepoint $x\in X$. We will need the following decay property for $S_x^G(gx,R) \subset G$.

\begin{proposition} \label{P:counting-decay}
There is a function $\rho : \mathbb{R}^+ \to \mathbb{R}^+$ with $\rho(r) \to 0$ as $r \to \infty$ such that for every $n \geq 0$
\[
P^n(S^G_x(gx, R)) \le \rho \left(d(x,gx) -R \right). 
\]
\end{proposition}

\begin{proof}
By Lemma \ref{lem:near_large_growth}, every $h \in S^G_x(gx, R)$ lies at distance $\leq D$ from a group element $h_1$ of large growth,
and moreover by Lemma \ref{lem:neigh-shadow} such an element $h_1$ is contained in the shadow $S^G_x(gx, R + c D)$, where $c$ is the Lipschitz constant 
of the orbit map $G \to X$. Thus, 
$$S^G_x(gx, R) \subseteq N_D(S^G_x(gx, R+cD) \cap LG).$$
Now, 
$$P^n(S^G_x(gx, R)) \leq P^n \left(N_D(S^G_x(gx, R+cD) \cap LG) \right) \leq $$
hence, if $\# B_D$ is the size of a ball of radius $D$ in the Cayley graph, 
$$\leq \#B_D \cdot P^n(S^G_x(gx, R+cD) \cap LG) \leq $$
and by Lemma \ref{lem:comparePn}
$$\leq c_1 \# B_D \cdot  \mathbb{P}^n(S^G_x(gx, R+cD)) \leq $$
hence by decay of shadows for $\mathbb{P}$ (Proposition \ref{p:decay}) 
$$\leq c_1 \# B_D \cdot p \left( d(x,gx) - R - cD \right)$$
thus the claim is proven if we set $\rho(r) := c_1 \# B_D \cdot p(r - cD)$.
\end{proof}

\subsection{Genericity of loxodromics}




For each $n$, let $n_1 = \lfloor \frac{n}{2} \rfloor$, $n_2 = n - n_1$. For each $g \in S_n$, let us pick its representative 
path from the initial vertex, and let $a$, $b$ be the group elements associated respectively  to the first $n_1$ edges, and the last $n_2$ 
edges of this path. Thus we can canonically write $g = ab$, with $a \in S_{n_1}$, $b \in S_{n_2}$, and $b$ in the cone of $a$.
We now show that $a$ and $b$ are almost independent: 

\begin{lemma} \label{l:indep-count}
There exists a constant $c > 0$ such that for any $n \geq 2$ the inequality
$$P^n(a \in A, b \in B) \leq c P^{n_1}(A) P^{n_2}(B)$$
is satisfied for any subsets  $A, B \subseteq G$.
\end{lemma}

\begin{proof}
$$P^n(a \in A, b \in B) = \frac{ \#\{(a,b) \in S_{n_1} \times S_{n_2} \ : \ b \in \cone(a), a \in A, b \in B \}}{\#S_n} \leq$$
and forgetting the requirement that $b \in \cone(a)$, 
$$\leq \frac{ \#\{(a,b) \in S_{n_1} \times S_{n_2} \ : \  a \in A, b \in B \}}{\#S_n} \leq \frac{P^{n_1}(A) P^{n_2}(B) \#S_{n_1} \#S_{n_2}}{\#S_n} .$$
Now, $\#S_{n_1} \#S_{n_2} \leq c \#S_n$ for some $c$ which depends only on $\Gamma$, proving the claim.
\end{proof}

\begin{lemma} \label{l:positivehalf}
For any $\epsilon > 0$, 
$$P^n\left(d(x,ax) \leq \frac{n(L+\epsilon)}{2}\right) \to 1$$
and 
$$P^n\left(d(x,bx) \leq \frac{n(L+\epsilon)}{2}\right) \to 1.$$
\end{lemma}

\begin{proof}
We prove the complementary statement that 
$$P^n\left(d(x,ax) \geq \frac{n(L+\epsilon)}{2}\right) \to 0.$$
Indeed, by Lemma \ref{l:indep-count}
$$P^n\left(d(x,ax) \geq \frac{n(L+\epsilon)}{2}\right) \leq c P^{n_1}\left(d(x, gx) \geq \frac{n(L+\epsilon)}{2} \right)$$
which tends to zero by Theorem \ref{th:gen_drift}, recalling that $n_1 \sim \frac{n}{2}$.
The proof of the second statement is completely analogous. 
\end{proof}

Our goal is to prove that the translation length of a generic element in the $n$--sphere grows linearly in $n$. In order to apply Lemma \ref{l:fellow_travel}, we need to check that the first half of $g$ (which is $a$) and the first half of $g^{-1}$ (which is $b^{-1}$) 
generically do not fellow travel: 

\begin{lemma} \label{l:prod0}
Let $f : \mathbb{R} \to \mathbb{R}$ be any function such that $f(n) \to +\infty$ as $n \to +\infty$. Then 
\[
P^n \Big( (a x, b^{-1} x)_x \geq f(n) \Big) \to 0
\]
as $n \to \infty$.
\end{lemma}

\begin{proof}
We compute
$$P^n \left((ax, b^{-1}x)_x \geq f(n)\right) = \frac{\#\{ (g,h) \in S_{n_1} \times S_{n_2} \ : \ h \in \cone(g), \ (gx, h^{-1}x)_x \geq f(n) \}}{\#S_n} \leq $$
and removing the requirement that $h \in \cone(g)$ we have
$$\leq  \frac{\#\{ (g,h) \in S_{n_1} \times S_{n_2} \ :  \ (gx, h^{-1}x)_x \geq f(n) \}}{\#S_n} \leq $$
$$\leq \frac{1}{\#S_n} \sum_{h \in S_{n_2}} \#\left\{ g \in S_{n_1}  \ :  \ gx \in S_x(h^{-1}x, d(x, h^{-1}x) - f(n)) \right \} \leq$$
and from decay of shadows (Proposition \ref{P:counting-decay}) follows that 
$$\leq \frac{1}{\#S_n} \sum_{h \in S_{n_2}} \rho(f(n)) \#S_{n_1} \leq \frac{ \#S_{n_1} \# S_{n_2} \rho(f(n))}{\#S_n} \leq c \rho(f(n)) \to 0 .$$
\end{proof}

The following two lemmas are the counting analogues of Lemma \ref{lem:prob_prod_2} and Lemma \ref{lem:prob_prod_3}.

\begin{lemma} \label{l:prod1}
For each $\eta > 0$, the probability
\[
P^n\left((b^{-1} x, g^{-1} x)_x \leq \frac{n (L - \eta)}{2} \right)
\]
tends to $0$ as $n \to \infty$.
\end{lemma}

\begin{proof}
By Lemma \ref{l:gromov}, and since the action is isometric
$$(b^{-1}x, g^{-1}x)_x \geq d(x, g^{-1}x) - d(b^{-1} x, g^{-1} x) = d(x, gx) - d(x, ax)$$
Now, for any $\epsilon > 0$, by genericity of positive drift (Theorem \ref{th:gen_drift})
$$P^n\left(d(x, gx) \geq n(L-\epsilon)\right) \to 1$$
and by Lemma \ref{l:positivehalf}
$$P^n\left(d(x, ax) \leq \frac{n}{2}(L+\epsilon)\right) \to 1$$
so 
$$P^n\left( (b^{-1}x, g^{-1}x)_x \geq n \left( \frac{L}{2} - \frac{3\epsilon}{2} \right) \right) \to 1$$
which proves the claim setting $\eta = 3 \epsilon/2$.

\end{proof}

\begin{lemma} \label{l:prod2}
For each $\eta > 0$, the probability
\[
P^n\left( (a x, g x)_x \leq \frac{n (L - \eta)}{2}  \right)
\]
tends to $0$ as $n \to \infty$.
\end{lemma}

\begin{proof}
By Lemma \ref{l:gromov}, 
$$(ax, gx)_x \geq d(x, gx) - d(ax, gx) = d(x, gx) - d(x, bx)$$
hence, since for any $\epsilon > 0$ we have by genericity of positive drift (Theorem \ref{th:gen_drift})
$$P^n(d(x, gx) \geq n(L-\epsilon)) \to 1$$
and by Lemma \ref{l:positivehalf}
$$P^n\left(d(x, bx) \leq \frac{n}{2}(L+\epsilon)\right) \to 1$$
we get 
$$P^n\left( (ax, gx)_x \geq n \left( \frac{L}{2} - \frac{3\epsilon}{2} \right) \right) \to 1$$
which proves the claim setting $\eta = 3 \epsilon/2$.
\end{proof}

The following proposition establishes control of the Gromov products $(gx,g^{-1}x)_x$ with respect to our counting measures:
\begin{proposition} \label{p:g-prod}
Let $f : \mathbb{N} \to \mathbb{R}$ be a function such that $f(n) \to +\infty$ as $n \to \infty$.
Then 
$$P^n \Big( (gx, g^{-1}x)_x \leq f(n) \Big) \to 1$$
as $n \to \infty$.
\end{proposition}

\begin{proof}
Define 
$$f_1(n) = \min \left\{ f(n) - 2 \delta, \frac{n(L-\eta)}{2} - 3\delta \right\}$$
It is easy to see that $f_1(n) \to \infty$ as $n \to \infty$.
By Lemma \ref{l:fellow_travel}, if we know that:
\begin{enumerate}
\item
 $(ax, gx)_x \geq n(L-\eta)/2$, 
 \item
 $(b^{-1}x, g^{-1}x)_x \geq n(L-\eta)/2$, and
\item
 $(ax, b^{-1}x)_x \leq f_1(n) \leq n(L-\eta)/2 - 3\delta$, 
\end{enumerate}
then 
$$(gx, g^{-1}x)_x \leq (ax, b^{-1}x)_x + 2 \delta \leq f_1(n) + 2\delta.$$
Using Lemmas \ref{l:prod0}, \ref{l:prod1}, and \ref{l:prod2}, the probability that conditions (1),(2), (3) hold tends to $1$, hence we have 
$$P^n( (gx, g^{-1}x)_x \leq f(n)) \to 1$$
as $n \to \infty$.
\end{proof}

Finally, we prove Theorem \ref{th:intro_gen_trans} from the introduction.
\begin{theorem}[Linear growth of translation length]
For any $\epsilon > 0$ we have  
\[
\frac{\#\{g \in S_n : \tau_X(g) \ge n(L-\epsilon) \}}{\#S_n} \to 1,
\]
as $n\to \infty$.
\end{theorem}

\begin{proof}
If we set $f(n) = \eta n$ with $\eta > 0$, then by Proposition \ref{p:g-prod} and Theorem \ref{th:gen_drift} the events 
$(gx, g^{-1}x)_x \leq \eta n$ and $d(x, gx) \geq n(L-\eta)$ occur with probability ($P^n$) which tends to $1$, hence 
by Lemma \ref{l:tau-formula} 
\begin{align*}
P^n \Big( \tau(g) \ge n(L- 3 \eta)  \Big) \ge 
P^n \Big ( d(x, gx) - 2(gx, g^{-1}x)_x + O(\delta) \ge n(L- 3 \eta) \Big) 
\end{align*}
which approaches $1$ as $n \to \infty$. This implies the statement if we choose $\epsilon > 3 \eta$.
\end{proof}

Since elements with positive translation length are loxodromic, we finally get

\begin{corollary}[Genericity of loxodromics]
\[
\frac{\#\{g \in S_n : g \; \mathrm{is} \; X - \mathrm{loxodromic} \}}{\#S_n} \to 1,
\]
as $n\to \infty$.
\end{corollary}





\bibliography{counting.bbl}
\bibliographystyle{amsalpha}

\end{document}